\documentclass{article}

\PassOptionsToPackage{numbers, compress}{natbib}

\usepackage[preprint]{neurips_2025}

\usepackage[utf8]{inputenc} 
\usepackage[T1]{fontenc}    
\usepackage{hyperref}       
\usepackage{url}            
\usepackage{booktabs}       
\usepackage{amsfonts}       
\usepackage{nicefrac}       
\usepackage{microtype}      
\usepackage{xcolor}         
\usepackage{xpatch}

\usepackage{xpatch}
\usepackage{algorithm}
\usepackage{algorithmic}
\usepackage{multirow} 
\usepackage{amsmath}
\usepackage{amsthm}
\usepackage{amssymb}

\newtheorem{theorem}{Theorem}[section]

\newtheorem{lemma}[theorem]{Lemma}

\newtheorem{assumption}[theorem]{Assumption}

\usepackage[textsize=tiny]{todonotes}

\DeclareMathOperator{\ReLU}{ReLU}
\DeclareMathOperator*{\argmin}{arg\,min}

\makeatletter
\xapptocmd{\NAT@bibsetnum}{\setlength{\leftmargin}{0pt}\setlength{\itemindent}{\labelwidth}\addtolength{\itemindent}{\labelsep}}{}{}
\makeatother

\title{Learning based convex approximation for constrained parametric optimization}

\author{%
  Kang Liu$^{\#}$ \\
  School of Future Technology\\
  Xi'an Jiaotong University\\
  Xi'an, 710049 \\
  \texttt{liukang@stu.xjtu.edu.cn} \\
  \And
  Wei Peng$^{\#}$ \\
  School of Future Technology\\
  Xi'an Jiaotong University\\
  Xi'an, 710049 \\
  \texttt{pengwei\_xjtu@163.com} \\
  \And
  Jianchen Hu\thanks{Corresponding Author} \\
  School of Automation Science and Technology\\
  Xi'an Jiaotong University\\
  Xi'an, 710049 \\
  \texttt{horace89@xjtu.edu.cn} \\
}

\begin{document}

\maketitle

\begin{abstract}
We propose an input convex neural network (ICNN)-based self-supervised learning framework to solve continuous constrained optimization problems. By integrating the augmented Lagrangian method (ALM) with the constraint correction mechanism, our framework ensures \emph{non-strict constraint feasibility}, \emph{better optimality gap}, and \emph{best convergence rate} with respect to the state-of-the-art learning-based methods. We provide a rigorous convergence analysis, showing that the algorithm converges to a Karush-Kuhn-Tucker (KKT) point of the original problem even when the internal solver is a neural network, and the approximation error is bounded. We test our approach on a range of benchmark tasks including quadratic programming (QP), nonconvex programming, and large-scale AC optimal power flow problems. The results demonstrate that compared to existing solvers (e.g., \texttt{OSQP}, \texttt{IPOPT}) and the latest learning-based methods (e.g., DC3, PDL), our approach achieves a superior balance among accuracy, feasibility, and computational efficiency.
\end{abstract}

\section{Introduction}
\label{sec:introduction}
Optimization problems widely arise in diverse domains, including the energy systems, logistics, healthcare, and many others. Based on their structural properties, these problems can be categorized as convex or nonconvex, and constrained or unconstrained~\cite{boyd2004convex}. In practice, many real-world optimization tasks are large-scale, nonconvex, and constrained. Typically, such problems lack the closed-form analytical solutions and cannot be solved efficiently by polynomial-time algorithms~\cite{nocedal2006numerical}. Consequently, the design of efficient numerical solvers remains a crucial research topic in optimization~\cite{bertsekas1999nonlinear,gurobi2024state}.

Current optimization solvers have made significant progress, enabling effective solutions in various application areas. The commercial closed-source solvers, such as \texttt{COPT}~\cite{copt} for linear optimization and \texttt{Gurobi}~\cite{gurobi} for mixed-integer and conic optimization, offer high accuracy and broad applicability. In open-source communities, \texttt{SCIP}~\cite{scip} excels in integer programming, \texttt{OSQP}~\cite{osqp} effectively solves convex quadratic problems, and \texttt{IPOPT}~\cite{ipopt} handles nonlinear optimization tasks. For power flow-related problems, tools such as \texttt{MATPOWER} and \texttt{PYPOWER}~\cite{matpower,pypower} are commonly employed. These solvers, collectively referred to as numerical iterative methods, perform well on problems with clear algebraic structures and moderate computational constraints.

Nevertheless, the traditional numerical iterative solvers inherently rely on the sequential computation, which poses a significant obstacle to parallel implementation~\cite{lu2023cupdlp}. Generally, an unconstrained optimization iteration adopts the following form~\cite{nocedal2006numerical}:
\begin{equation}
x^{k+1} = x^k + \alpha^k p^k,
\end{equation}
where $x^k$ denotes the current solution iterate, $\alpha^k$ is the step size selected by line search strategies (e.g., golden section or Armijo rule~\cite{armijo1966minimization}), and $p^k$ represents a descent direction (e.g., gradient or Newton-like direction). For constrained optimization problems, methods like penalty-based reformulation or Lagrangian dualization are commonly employed to transform constrained problems into unconstrained forms~\cite{nocedal2006numerical}.

Parallelization efforts in optimization mainly fall into two categories: model-level decomposition and gradient-level parallelism. The model-level decomposition (such as Alternating Direction Method of Multipliers (ADMM)~\cite{boyd2011distributed} and logic-based Benders decomposition~\cite{hooker2003logic}) partitions the original problem into structured subproblems. However, these methods require specific structural properties such as variable decoupling, sparsity, or diagonal dominance, thus restricting their application range. On the other hand, the gradient-level parallelization leverages independent components of decision variables processed simultaneously, as observed in mini-batch stochastic gradient descent~\cite{qian2020impact} and block coordinate descent~\cite{nutini2017let}. Even though the computational efficiency of each iteration can be improved, the total iteration numbers are not reduced and the convergence instability can be encountered for the gradient-level approaches.

Recently, deep learning methods have garnered increased attention due to their rapid inference speed and inherent parallelism. Among learning-based approaches, supervised learning represents a straightforward strategy for optimization~\cite{Fioretto2020Predicting,Pan2020DeepOPF,He2023Fast}. This method generates datasets comprising pairs of parameterized problem instances $\theta_i$ and their optimal solutions $x_i^*$, then train the neural networks to directly predict solutions by minimizing the prediction loss (e.g., $L = \|x_i^* - x_i'\|_2$). Although the neural network inference is computationally efficient, these predicted solutions do not inherently satisfy constraints, which limits their practical applicability in constrained optimization.

In order to recover the constraint feasibility, hybrid \emph{predict-and-correct} approaches including DC3~\cite{donti2021dc3} and PDL~\cite{park2022pdl} have appeared. These approaches leverage the self-supervised learning to obtain an initial prediction and subsequently correct it through iterative correction based on the gradient of the constraint violations, or Karush–Kuhn–Tucker (KKT) conditions. The efficacy of these methods heavily relies on the quality of the initial guess---a poor initial prediction necessitates more extensive corrections, potentially slowing or obstructing convergence.

The correction procedure usually follows either the gradient-based or neural network-based paradigms~\cite{luken2024self}. The gradient-based correction is inherently non-parallelizable. Conversely, the neural network-based correction is parallelizable but includes great memory consumption and may produce infeasible solution. Thus, the neural network-based correction is scalable and suitable for large-scale nonconvex problems, e.g., the neural prediction of descent directions~\cite{garrigos2023adapting}, differentiable optimization via end-to-end backpropagation~\cite{pan2024bpqp}, reinforcement learning-driven branching strategies in branch-and-bound algorithms~\cite{gasse2019reinforcement}, and graph representation learning methods tailored to combinatorial optimization~\cite{bengio2023solving}. Several such methodologies have already been integrated into state-of-the-art solvers, including \texttt{Gurobi}.

In this work, we propose a novel self-supervised learning framework based on Input Convex Neural Networks (ICNN) for solving general continuous optimization problems. ICNN can achieve a convex mapping with positive weight and ``passthrough'' channel, thus beneficial for the convergence and stability. By incorporating ICNN-based correction into the augmented Lagrangian framework, we achieve significant improvements in parallelizability and scalability, while maintaining strong guarantees on solution feasibility and convergence. Our contributions are, 

i) We propose an ICNN-based self-supervised learning framework which includes: 1) \emph{non-strict constraint feasibility}; 2) \emph{better optimality gap}; 3) \emph{fastest convergence rate} with respect to current learning-based methods in terms of solving the constrained optimization problem.

ii) We provide a rigorous theoretical analysis, proving that the augmented Lagrangian multiplier (ALM) with neural network proxy solver can converge to a KKT point and the approximation error is bounded.

The remainder of this paper is structured as follows. Section~\ref{sec:related work} reviews related literatures. Section~\ref{sec:preliminary} is the preliminary.
Section~\ref{sec:method} describes our proposed approach.
Section~\ref{sec:analysis} provides the theoretical analysis.
Section~\ref{sec:experiments} presents empirical evaluations and extensive comparisons with existing approaches, and Section~\ref{sec:conclusion} concludes the paper.

\section{Related Work}\label{sec:related work}
In the constrained optimization, the Lagrangian multiplier method~\cite{powell1969method, hestenes1969multiplier} is one of the fundamental approaches which can transform the constrained problems into unconstrained ones by incorporating the Lagrangian function. In order to recover the strong convex property and improve the convergence, the ALM method is provided by integrating penalty functions \cite{alm}. 

In addition to the classical optimization methods, the machine learning techniques have made significant progress in the field of optimization problems approximation and acceleration, forming two technical paradigms~\cite{BENGIO2021405}. The first category is learning-to-optimize, which embeds machine learning into optimization models to improve the heuristic strategies or branch-and-bound rules~\cite{milp, search}. The second category is optimization proxy, which constructs end-to-end machine learning models to directly map problem parameters to complete solutions~\cite{graphs, fast}. The reinforcement learning is suitable for the first category, applicable to optimization node selection and learning local search rules, making it more appropriate for learnable optimization tasks~\cite{reinforcement}. The supervised learning is suitable for the second category, particularly superior for solving black-box optimization problems~\cite{article}. However, the supervised learning relies on pre-solved instance pairs from historical data, which is unrealistic to obtain since there may not exist a well-tuned exact solver for the problem. The self-supervised learning method is developed to overcome this limitation, e.g., DC3~\cite{donti2021dc3} and PDL~\cite{park2022pdl} methods use the optimization parameters and autmented Lagrangian function to train the neural network directly.

For constrained optimization problems, effectively enforcing constraints in neural networks has become a key research focus. The main technical approaches include: 1) adding the constraint violation penalty~\cite{dener2020training}; 2) alternating the optimizer during training~\cite{ma}; 3) designing new neural network architectures to ensure the constraint feasibility (such as DeepSaDe~\cite{Goyal2023DeepSaDeLN} and ConstraintNet~\cite{Brosowsky2020SampleSpecificOC}). Different from the above works, we first propose an error-bounded self-supervised learning based approach for constrained optimization. We employ ALM-based loss function to learn the mapping from parameter $x$ to optimal solution $y^*(x)$. Throughout training, we exploit structural information inherent to the problem, to guide network parameter updates, thereby approximating the set-valued solution mapping $S(x)$. 

\section{Preliminary}\label{sec:preliminary}

\subsection{The Constrained Optimization Problem}

We consider the following general form of continuous constrained optimization problem:
\begin{equation}\label{eq:general_opt}
\begin{aligned}
\min_{y\in \mathcal{Y}} ~f_x(y),\quad  
\text{s.t.} \quad  g_x(y)\leq 0, \quad h_x(y)=0.
\end{aligned}
\end{equation}
where $x\in \mathbb{R}^d$ denotes problem-specific parameters defining a particular instance of the optimization task, $y\in \mathbb{R}^n$ represents the solution of this instance, and $\mathcal{Y}$ is the feasible region. $f_x(y)$, $g_x(y)$, and $h_x(y)$ denote the objective function, inequality constraints, and equality constraints, respectively.

\subsection{Parametric Optimization and Set-Valued Mappings}
When the research interest focuses on the variations with respect to the problem parameter $x$, we encounter the parametric optimization formulation. The optimal value function is defined as
\begin{equation}\label{eq:opt_val_func}
v(x)=\min_{y\in Y(x)} f(x,y),
\end{equation}
where the feasible set $Y(x)$ for given parameter $x$ is given by
\begin{equation}\label{eq:feasible_set}
Y(x)=\{y\in \mathcal{Y}\mid g_x(y)\leq 0,\, h_x(y)=0\}.
\end{equation}

Correspondingly, the optimal solution mapping (set-valued mapping) is defined by
\begin{equation}\label{eq:solution_mapping}
S(x)=\{y\in Y(x)\mid f(x,y)=v(x)\}.
\end{equation}
\subsection{Self-Supervised Learning for Optimization}

Conventional optimization solvers require significant computational resources, particularly in large-scale problems. Self-supervised learning leverages the intrinsic structures of the optimization problem to generate supervisory signals without explicit manual labeling.

\section{Methodology}\label{sec:method}

This section elaborates the proposed methodology, encompassing the ALM construction, the ICNN, as well as method for correcting feasible solutions.

\subsection{The ALM Construction}
We design a neural network architecture, referred to as Augmented Lagrangian Network (ALN), to emulate the ALM procedure. The ALN is trained via the augmented Lagrangian loss function, which is defined as follows,
    \begin{equation}\label{eq:alm_loss}
    \small
    \begin{aligned}
    L_p\left(y\mid \mu,\lambda\right)&=f_x(y)+\mu^T \operatorname{ReLU}(g_x(y))+\lambda^T h_x(y)+\frac{\rho}{2}\left(\sum_{j\in G}v(g_{x,j}(y))+\sum_{j\in H}v(h_{x,j}(y))\right),
    \end{aligned}
    \end{equation}
where $\operatorname{ReLU}(z)=\max(z,0)$, which handles the inequality constraints $g_x(y)$. Specifically, $\operatorname{ReLU}(g_x(y))$ equals zero for $g_x(y)\le 0$, and returns positive values when constraints are violated, thus applying penalties solely on constraint violations. Here, $\mu$ and $\lambda$ are the dual multipliers; $\rho$ denotes the penalty coefficient; sets $G$ and $H$ denote violated inequality and equality constraints.

The penalty functions are defined as:
    \begin{equation}\label{eq:eqieq_penalty}
    v(g(y))=(\operatorname{ReLU}(g(y)))^2,\quad
    v(h(y))=(h(y))^2.
    \end{equation}
These quadratic penalty terms significantly amplify the penalty for severe constraint violations. By this construction, the penalties vanish whenever constraints are satisfied ($g_x(y)\leq 0$ and $h_x(y)=0$), while they increase quadratically with the increase of constraint violation. 

In the ALM procedure, each iteration comprises three key steps: (i) minimization of the augmented Lagrangian function, (ii) updating dual multipliers, and (iii) adjusting the penalty parameter. Step (i) corresponds to the inner iteration, implemented by gradient descent and backpropagation through the ALN model to approximate optimal solutions for given parameters $x$; Step (ii) and (iii) constitute the outer iteration, with the core objective of gradually updating dual variables and penalty coefficients to enhance constraint satisfaction and optimization stability.

Let $y=\Pi_\theta(x)$ denote the ALN output at iteration $k$, with corresponding dual variables $\mu_k$ and $\lambda_k$. The outer iteration update rules are given by
\begin{equation}\label{eq:lambda,mu_update}
    \mu_{k+1}\leftarrow\max\left\{\mu_k+\rho\max\left(g_x(y),0\right),\,0\right\},\quad\lambda_{k+1}\leftarrow\lambda_k+\rho h_x(y).
\end{equation}
Let us consider dual variable estimates, an auxiliary measure is defined as \cite{luken2024self}
    \begin{equation}\label{eq:max_ieq_vio}
    \sigma_{x,j}(y)=\max\left\{g_{x,j}(y),-\frac{\mu_{k,j}}{\rho}\right\},\quad\forall j\in G.
    \end{equation}
Consequently, the maximal violation degree $\nu_k$ at iteration $k$ is defined as
    \begin{equation}\label{eq:max_constr_vio}
    \nu_k=\max\left\{\|h_x(y)\|_{\infty},\,\|\sigma_x(y)\|_{\infty}\right\}.
    \end{equation}
As constraint violation gradually decreases, penalty terms reduce accordingly. To maintain convergence stability and ensure a stronger penalty when constraints continue being violated, ALN dynamically adjusts the penalty coefficient $\rho$. Specifically, the updating rule for $\rho$ when $k>1$ and $\nu_k>\tau \nu_{k-1}$ is
    \begin{equation}\label{eq:rho_update}
    \rho\leftarrow\min\{\alpha\rho,\rho_{\max}\},
    \end{equation}   
where $\tau\in(0,1)$ represents a tolerance parameter, $\alpha>1$ denotes a growth factor, and $\rho_{\max}$ is the maximum penalty value to prevent numerical instabilities.
Combining these steps, we get the Algorithm\ref{alg:ALM}.

\begin{small}
    \begin{algorithm}[htbp]
   \caption{ALN Training Procedure}
   \label{alg:ALM}
\begin{algorithmic}[1]
   \STATE {\bfseries Input:} Initial penalty parameter $\rho>0$, growth factor $\alpha>1$, tolerance factor $\tau\in(0,1)$, penalty upper bound $\rho_{\max}$, initial dual multipliers $\mu_0\geq 0$, $\lambda_0$, neural network model $\Pi_{\theta}(x)$, maximum iterations $K$, dataset $\mathcal{D}=\{x_i\}_{i=1}^{N}$.
   \STATE {\bfseries Output:} Trained network parameters $\theta$.
   \FOR{$k=1,2,\dots,K$}
       \FOR{each training instance $x_i\in\mathcal{D}$}
           \STATE $y_i\leftarrow\Pi_{\theta}(x_i)$
           \STATE Compute constraint violations using \eqref{eq:max_ieq_vio} and \eqref{eq:max_constr_vio}
           \STATE Compute ALM loss $L_p(y_i|\mu_{k-1,i},\lambda_{k-1,i})$ using \eqref{eq:alm_loss}
           \STATE Update neural network parameters $\theta$ by gradient descent on $L_p$
           \STATE Update dual multipliers using \eqref{eq:lambda,mu_update}
       \ENDFOR
       \STATE Evaluate maximal violation using \eqref{eq:max_constr_vio}
       \IF{$k>1$ and $\nu_k>\tau \nu_{k-1}$}
           \STATE update penalty coefficient using \eqref{eq:rho_update}
       \ENDIF
       \IF{converged}
           \STATE {\bfseries break}
       \ENDIF
   \ENDFOR
   \STATE \textbf{Correct} the solution $y$ via constraint satisfaction correction procedure given in Section \ref{sec:constraint_correction}.
\end{algorithmic}
\end{algorithm}
\end{small}

\textit{Note on KKT Convergence Check:} Convergence could alternatively be assessed via the KKT residual $R_{KKT}^{(k)} \le \epsilon_{KKT}$, where $\epsilon_{KKT} > 0$ is a tolerance. The residual combines primal feasibility ($\nu_k$ from Eq.~\eqref{eq:max_constr_vio}) and stationarity (Lagrangian gradient norm):
\begin{equation*}
R_{KKT}^{(k)} := \max \left\{ \nu_k, \, \|\nabla_y \mathcal{L}(y^k, \mu^k, \lambda^k)\|_{\infty} \right\}
\end{equation*}
Here, $\mathcal{L}(y, \mu, \lambda) = f_x(y) + \mu^T g_x(y) + \lambda^T h_x(y)$ is the standard Lagrangian. (Algorithm~\ref{alg:ALM} uses validation-based early stopping).

\subsection{ICNN Principle}
In the previous section, we illustrate the parameter update rules in the outer iteration. The efficiency of the inner iteration heavily depends on the neural network design. Instead of the multi-layer perceptron (MLP), we adopt ICNN. 

\subsubsection{Network Structure and Construction}

The primary distinction of ICNN from MLP is that it is a convex mapping, guaranteed by Lemma~\ref{lemma:convex-composition}.

\begin{lemma}\label{lemma:convex-composition}\cite{amos2017input}
Let $h(x)$ be convex and non-decreasing, and $u(y)$ be convex. Then, the composite function $h(u(y))$ is convex.
\end{lemma}

Consider a $k$-layer ICNN, with input $y$ and output $z_k$. Denote the network's parameters (weights and biases) by $\theta=\{W^{(y)}_{0:k-1},W^{(z)}_{1:k-1},b_{0:k-1}\}$. The activation at layer $i+1$ is given by:
    \begin{equation}\label{eq:icnn_layer}
    z_{i+1} = 
    \begin{cases}
    g_i\left(W_i^{(z)} z_i + W_i^{(y)} y + b_i\right), & i=0,1,\dots,k-2, \\[6pt]
    f(y;\theta), & i=k-1.
    \end{cases}
\end{equation}
The ICNN structure is illustrated in Figure~\ref{fig:icnn_structure}.
\begin{figure}[htbp]
  \centering
  \includegraphics[width=0.6\linewidth]{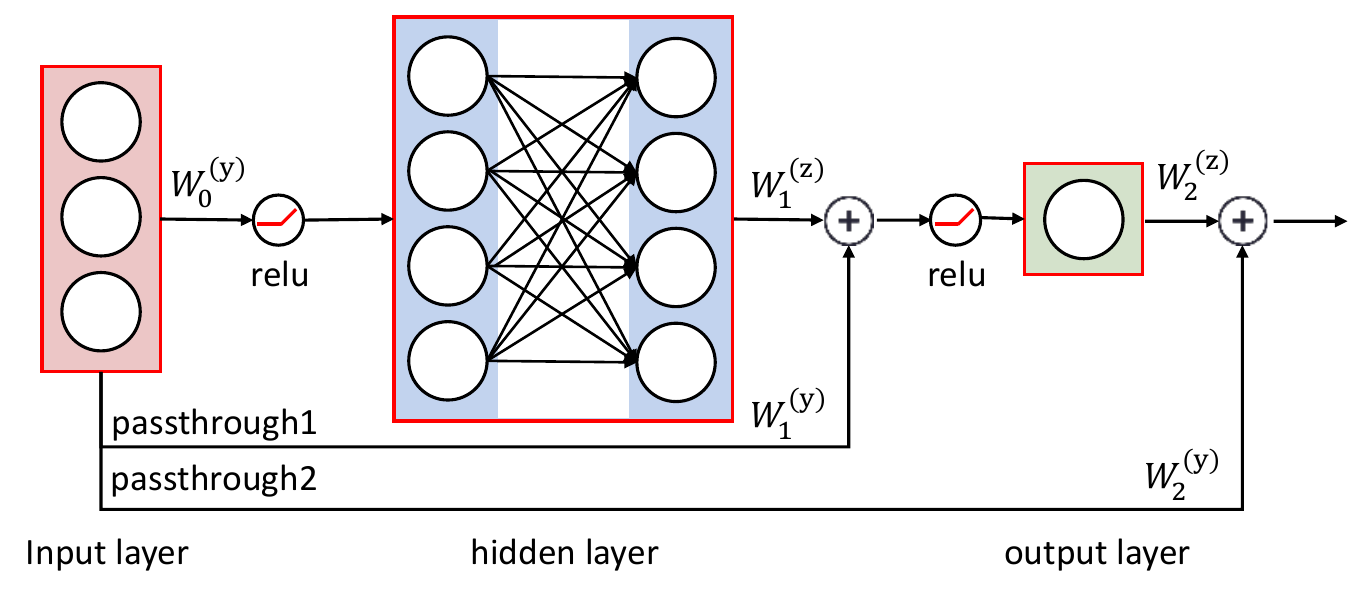}
  \caption{Illustration of ICNN architecture.}
  \label{fig:icnn_structure}
\end{figure}
Additionally, direct connections from the input to hidden layers, termed ``passthrough'' connections, are not constrained by non-negativity and may have arbitrary weights and biases. These passthrough connections significantly enhance network expressiveness. Compared to standard MLPs, ICNNs are less susceptible to local minima during backpropagation, resulting in faster and more stable convergence, requiring fewer iterations to reach high-quality solutions. The proof of $f(y;\theta)$'s convexity is provided in~\cite{amos2017input}.

\subsection{Constraint Satisfaction Correction}
\label{sec:constraint_correction}

In order to enhance the feasibility of solutions generated by a primary method (e.g., a neural network prediction), we introduce an iterative correction procedure. The motivation is to systematically reduce constraint violations by taking steps guided by the gradients of infeasibility measures. This process iteratively projects or moves an initial solution estimate $y^{(0)}$ closer to the feasible set. We outline two approaches tailored to different structures of the equality constraints.

\subsubsection{Correction with Explicit Equality Constraint Parametrization}
\label{sec:correction_implicit_balanced}

This approach is employed when the equality constraints $h(x, y) = 0$ can be resolved by defining $q$ dependent variables $w \in \mathbb{R}^q$ as an explicit or implicit function of $n-q$ independent variables $z \in \mathbb{R}^{n-q}$, denoted $w = \varphi_x(z)$. The existence and differentiability of $\varphi_x$ are often underpinned by the Implicit Function Theorem, applicable when the Jacobian $\nabla_w h$ is non-singular. This technique, also observed in methods like~\cite{donti2021dc3}, ensures that any point $y = (z, \varphi_x(z))$ inherently satisfies $h(x, y) = 0$.

With equality constraints handled by the parametrization, correction efforts concentrate on satisfying the inequality constraints $g(x, y) \leq 0$. We achieve this by performing gradient descent on a measure of inequality violation, defined in terms of the independent variables $z$ as
    \begin{equation}
    V_g(z; x) = \frac{1}{2} \|\ReLU(g(x, (z, \varphi_x(z))))\|_2^2.
    \end{equation}
The update step adjusts $z$ in the negative gradient direction of this violation:
\begin{equation}
z^{(k+1)} = z^{(k)} - \gamma \nabla_z V_g(z^{(k)}; x),
\end{equation}   
where $\gamma > 0$ is a step size. Computing $\nabla_z V_g$ involves the chain rule, requiring gradients of $g$ and the Jacobian of $\varphi_x$. The dependent variables are subsequently updated using $w^{(k+1)} = \varphi_x(z^{(k+1)})$ to ensure the refined point $y^{(k+1)}$ remains on the equality constraint manifold.

\subsubsection{Correction via Joint Constraint Handling}
\label{sec:correction_joint_balanced}

When an explicit parametrization $\varphi_x$ is impractical or unavailable, correction is applied directly to the full decision vector $y \in \mathbb{R}^n$. This method addresses both inequality and equality constraints simultaneously by minimizing a composite violation measure:
\begin{equation}
V(y; x) = \frac{1}{2} \|\ReLU(g(x, y))\|_2^2 + \frac{\alpha}{2} \|h(x, y)\|_2^2.
\end{equation}
Here, $\alpha > 0$ serves as a weighting factor determining the relative emphasis on reducing equality versus inequality violations. The correction step follows the negative gradient of this composite measure:
\begin{equation}
y^{(k+1)} = y^{(k)} - \gamma \nabla_y V(y^{(k)}; x).
\end{equation}
Calculating $\nabla_y V$ involves the gradients (or Jacobians) of both $g$ and $h$. A subgradient is typically used for the ReLU term at points of non-differentiability.

\subsubsection{Iterative Application}
\label{sec:correction_iter_balanced}

Both correction approaches define a single step $y^{(k+1)} = C_x(y^{(k)})$ designed to decrease constraint violation. To achieve substantial improvement in feasibility, this step is typically applied iteratively starting from an initial estimate $y^{(0)}$. Thus, the refined solution after $T$ steps is $y^{(T)} = C_x^T(y^{(0)})$.

The motivation for iteration stems from the gradient descent nature of the step $C_x$; repeated application progressively minimizes the violation measure ($V_g$ or $V$), thereby driving the iterates towards the feasible set, resulting in a \emph{non-strict constraint feasibility}, as analyzed in related works on projection or feasibility restoration methods \cite{busseti2019solution, lee2017first}. In practical implementations, a fixed number of iterations $T$ is chosen. This allows for consistent application during both training (where the steps might be unrolled for backpropagation) and testing, providing a predictable computational footprint for the correction stage.

\section{Convergence and Stability Analysis}
\label{sec:analysis}

This section provides a theoretical analysis of the proposed ALM variant, where the inner optimization subproblem is approximately solved using ICNNs. We establish the conditions under which this approach converges to a KKT point of the original constrained optimization problem, thereby providing theoretical justification for using ICNNs as proxy solvers within the ALM framework. 

\subsection{Convergence}
The convergence analysis relies on the following standard assumptions:

\begin{assumption}[Feasible region]
\label{assum:feasible_main}
For each fixed parameter $x$, the feasible region $Y(x)=\{y\in\mathcal{Y}: g(x,y)\leq 0, h(x,y)=0\}$ is non-empty and compact.
\end{assumption}

\begin{assumption}[Function properties]
\label{assum:functions_main}
The objective function $f(x,y)$ and constraint functions $g(x,y)$, $h(x,y)$ are continuously differentiable with respect to $y$.
\end{assumption}

\begin{assumption}[Constraint qualification]
\label{assum:constraint_qual_main}
For any feasible point $y \in Y(x)$, appropriate constraint qualification conditions (e.g., Mangasarian-Fromovitz Constraint Qualification or Linear Independence Constraint Qualification hold.
\end{assumption}


The primary convergence result is stated in the following theorem:

\begin{theorem}[Convergence to KKT point]
\label{thm:aln_to_kkt_main}
Under Assumptions \ref{assum:feasible_main}-\ref{assum:constraint_qual_main}, the sequence generated by the ALM algorithm using an ICNN proxy solver for the inner minimization converges to a KKT point of the original optimization problem, provided the inner approximation error diminishes over iterations and the penalty parameters are appropriately updated.
\end{theorem}

\begin{proof} The proof of Theorem \ref{thm:aln_to_kkt_main} involves the following two key steps (detailed in Appendix~\ref{app:convergence_proofs}):

\textbf{1. ICNN Inner Approximation Accuracy:} We first demonstrate that the ICNN can approximate the solution of the inner ALM subproblem with arbitrary precision. Theorem~\ref{thm:icnn_accuracy_app} (Appendix \ref{app:sub:icnn_accuracy}) formally establishes that for any desired accuracy $\epsilon > 0$, an ICNN can be found such that the objective value achieved by its output $y_k$ is within $\epsilon$ of the true minimum of the augmented Lagrangian $L_\rho$. This relies on the universal approximation capability of ICNNs for continuous functions on compact sets and Lemma~\ref{lemma:uniform_approx_app} (Appendix \ref{app:sub:icnn_accuracy}), which links uniform function approximation to objective value approximation.

\textbf{2. Inexact ALM Convergence:} Building upon the adequate approximation of the inner loop, we analyze the convergence of the outer ALM iterations (multiplier and penalty parameter updates). Theorem~\ref{thm:inexact_alm_app} (Appendix \ref{app:sub:multiplier_conv}) shows that even with inexact solutions $y^k$ from the ICNN (provided the approximation error $\epsilon_k \to 0$), the overall ALM sequence possesses limit points that satisfy the KKT conditions of the original problem. The proof involves showing the boundedness of the iterate sequence $\{y^k\}$ (due to Assumption \ref{assum:feasible_main}) and the multiplier sequences $\{\mu^k, \lambda^k\}$, and then taking the limit to verify the KKT conditions. A careful analysis (Appendix \ref{app:sub:relu_kkt}) confirms that the use of ReLU in the multiplier updates correctly enforces complementarity slackness in the limit.
The detailed, step-by-step proofs for these intermediate results (Lemma~\ref{lemma:uniform_approx_app}, Theorem~\ref{thm:icnn_accuracy_app}, Theorem~\ref{thm:inexact_alm_app}, Lemma~\ref{lemma:relu_app}) leading to Theorem \ref{thm:aln_to_kkt_main} are provided in Appendix~\ref{app:convergence_proofs}.
\end{proof}

\subsection{Stability}
Beyond convergence, we analyze the suitability of ICNNs for approximating the optimal solution mapping $y^*(x)$ which implicitly relates the problem parameters $x$ to the optimal solution $y$.

\textbf{1. Continuity of Solution Mapping:} We first establish (see Theorem~\ref{thm:berge_corollary_app}, Appendix~\ref{app:effectiveness_proofs}) that under mild conditions (continuity of functions, compactness of feasible sets, and upper semicontinuity of the feasible set mapping), the optimal solution mapping $S(x)$ is upper semicontinuous. This implies that small changes in parameters $x$ lead to small changes in the set of optimal solutions, making the mapping potentially learnable by a neural network.

\textbf{2. ICNN Approximation Properties:}
When the true optimal solution mapping $y^*(x)$ happens to be convex (e.g., in certain linear-strongly convex problems, see Theorem~\ref{thm:convex_mapping_app}, Appendix~\ref{app:effectiveness_proofs}), ICNNs, being universal approximators for convex functions, can approximate $y^*(x)$ with arbitrary accuracy.

When $y^*(x)$ is non-convex, an ICNN, by its structure, inherently outputs a convex function. Training an ICNN to approximate a non-convex $y^*(x)$ is equivalent to finding the best convex approximation in a least-squares sense (see Theorem~\ref{thm:icnn_convex_regression_app}, Appendix~\ref{app:effectiveness_proofs}).

The convergence analysis (Theorem \ref{thm:aln_to_kkt_main}) provides crucial theoretical support for our method. It guarantees that despite using a neural network approximator (ICNN) within the ALM loop, the algorithm can still find theoretically valid solutions (KKT points) under reasonable assumptions. The effectiveness analysis highlights that ICNNs are well-suited when the underlying solution mapping is expected to be convex or when a stable, globally optimal \textbf{convex} approximation of a potentially complex, non-convex mapping is desired. This structural property of ICNNs can lead to more stable and predictable behavior compared to unconstrained approximators like standard MLPs, especially when dealing with optimization problems where convexity plays a role. The convex approximation property (Theorem \ref{thm:icnn_convex_regression_app}) ensures that even for non-convex mappings, the ICNN finds the best fit within the well-behaved class of convex functions.

\section{Experiments and Result Analysis}\label{sec:experiments}

We evaluate the performance of our proposed method \textbf{AIC} on four representative tasks: \textbf{QP}, \textbf{nonconvex programming}, \textbf{ACOPF 57-bus}, and large-scale \textbf{ACOPF 118-bus}. We compare AIC with the following baselines (abbreviated in the table):

\begin{itemize}
    \item \textbf{Solvers}: For QP, we use \texttt{OSQP}~\citep{osqp} and \texttt{qpth}~\citep{amos2019qpth}; for nonconvex programming, we use \texttt{IPOPT}~\citep{ipopt}; for ACOPF problems, we use \texttt{PYPOWER}~\citep{pypower}.
    \item \textbf{DC3}~\citep{donti2021dc3}: A self-supervised learning method using the Lagrangian objective as a loss function, with equality completion and gradient correction.
    \item \textbf{PDL}~\citep{park2022pdl}: A primal-dual learning architecture mimicking ALM with two networks: PrimalNet and DualNet.
\end{itemize}

We also conduct ablation studies to evaluate each module in AIC:
\begin{itemize}
    \item \textbf{ALN}: A DNN that learns to approximate the augmented Lagrangian function.
    \item \textbf{ALN w/ grad}: ALN combined with corrections for equality and inequality constraints.
    \item \textbf{ALN w/ cvx}: Replaces DNN with ICNN.
    \item \textbf{AIC}: Our full method. All variables are directly predicted by an ICNN, with post-correction via constraint-guided gradient steps.
\end{itemize}

In order to ensure fair comparison between DNN and ICNN under the same model complexity, we fix the neural network architecture across all experiments: two hidden layers with 500 units, ReLU activation, and one output layer. Gradient correction is performed for both equality and inequality constraints. For QP and Nonconvex optimization problem tasks, we set $t_{\text{train}} = t_{\text{test}} = 10$; for ACOPF-57 and ACOPF-118, we set $t_{\text{train}} = t_{\text{test}} = 5$.

All experiments are conducted on a single GeForce RTX 4090 GPU. We report wall-clock time averaged over test instances, including both optimization and inference time.

\subsection{Quadratic Programs}

We first evaluate on QP problems with linear constraints:
\begin{equation}
\begin{aligned}
\min_{y} \quad \frac{1}{2} y^\top Q y + p^\top y \quad \text{s.t.} \quad  A y = x, \quad G y \le h.
\end{aligned}
\end{equation}
where $Q \in \mathbb{R}^{n \times n}$ is positive semidefinite, $p \in \mathbb{R}^n$, $A \in \mathbb{R}^{n_{\text{eq}} \times n}$, $G \in \mathbb{R}^{n_{\text{ineq}} \times n}$, and $h \in \mathbb{R}^{n_{\text{ineq}}}$. The input $x$ is sampled uniformly from $[-1,1]^n$.

In Table~\ref{tab:convex_qp}, we compare AIC and all baselines on the case $n_{\text{eq}} = n_{\text{ineq}} = 50$. Appendix B.3.1 includes results for additional settings with $n_{\text{eq}}, n_{\text{ineq}} \in \{10,30,70,90\}$ (fixed $n=100$). We generate 10,000 problem instances and split them into train/val/test with a 10:1:1 ratio. Each method is run 5 times with tuned hyperparameters for fair performance comparison.

The first two rows of Table~\ref{tab:convex_qp} present results from \texttt{OSQP} and \texttt{qpth}, used as oracle references. ``Obj. value'' reports the average objective over the test set; ``Max/Mean eq.'' and ``Max/Mean ineq.'' report the maximum and average violation of equality and inequality constraints, respectively.

\begin{table}[t]
\centering
\scriptsize
\caption{Results on QP with $n=100$, $n_{\text{eq}} = n_{\text{ineq}} = 50$. Reported: objective, max/mean constraint violation, and runtime (mean over 5 runs).}
\label{tab:convex_qp}
\begin{tabular}{lcccccc}
\toprule
\textbf{Method} & \textbf{Obj. value} & \textbf{Max eq.} & \textbf{Mean eq.} & \textbf{Max ineq.} & \textbf{Mean ineq.} & \textbf{Time (s)} \\
\midrule
Solver (OSQP) & -15.047 (0.000) & 0.000 (0.000) & 0.000 (0.000) & 0.001 (0.000) & 0.000 (0.000) & 0.003 (0.000) \\
Solver (qpth) & -15.047 (0.000) & 0.000 (0.000) & 0.000 (0.000) & 0.000 (0.000) & 0.000 (0.000) & 1.509 (0.487) \\
\midrule
ALN & -14.967 (0.005) & 0.009 (0.001) & 0.003 (0.000) & 0.001 (0.000) & 0.000 (0.000) & 0.001 (0.000) \\
ALN w/ grad & -15.027 (0.003) & 0.004 (0.001) & 0.001 (0.000) & 0.003 (0.001) & 0.000 (0.000) & 0.006 (0.000) \\
ALN w/ cvx & -15.028 (0.001) & 0.004 (0.001) & 0.001 (0.000) & 0.001 (0.000) & 0.000 (0.000) & 0.002 (0.000) \\
AIC & \textbf{-15.036 (0.001)} & \textbf{0.002 (0.000)} & \textbf{0.001 (0.000)} & \textbf{0.001 (0.000)} & \textbf{0.000 (0.000)} & 0.004 (0.000) \\
DC3 & -13.448 (0.023) & 0.000 (0.000) & 0.000 (0.000) & 0.000 (0.000) & 0.000 (0.000) & 0.202 (0.044) \\
PDL & -15.017 (0.009) & 0.005 (0.000) & 0.002 (0.000) & 0.001 (0.000) & 0.000 (0.000) & -- \\
\bottomrule
\end{tabular}
\end{table}

\vspace{0.5em}
Compared to self-supervised baselines, AIC significantly outperforms PDL and achieves better feasibility and optimality than DC3, which enforces hard constraints but yields suboptimal solutions. Compared to solvers, AIC is approximately 700$\times$ faster than \texttt{qpth} and runs at similar speed to \texttt{OSQP}, with the added benefit of GPU-based batch parallelization, which traditional solvers lack.

\subsection{Nonconvex programs}

We further evaluate AIC on a nonconvex program defined as:
\begin{equation}
\min_{y \in \mathcal{Y}} \; \frac{1}{2} y^\top Q y + p^\top \sin(y) \quad \text{s.t.} \quad A y = x,\quad G y \le h.
\end{equation}
Table~\ref{tab:nonconvex_qp} compares the performance of AIC and self-supervised baselines with the classical non-convex solver \texttt{IPOPT}~\citep{ipopt}. AIC achieves a competitive objective while maintaining feasibility. Notably, DC3 satisfies all constraints but exhibits a large optimality gap (8.02\%), and although it runs faster than IPOPT per instance, its total inference time is significantly higher. AIC, in contrast, is approximately 20$\times$ faster than IPOPT and 70$\times$ faster than DC3.

\begin{table}[t]
\centering
\scriptsize
\caption{Performance on nonconvex programming with $n=100$, $n_{\text{eq}} = n_{\text{ineq}} = 50$.}
\label{tab:nonconvex_qp}
\begin{tabular}{lcccccc}
\toprule
\textbf{Method} & \textbf{Obj. value} & \textbf{Max eq.} & \textbf{Mean eq.} & \textbf{Max ineq.} & \textbf{Mean ineq.} & \textbf{Time (s)} \\
\midrule
Solver (IPOPT) & -11.592 (0.000) & 0.000 (0.000) & 0.000 (0.000) & 0.000 (0.000) & 0.000 (0.000) & 0.097 (0.004) \\
ALN               & -11.515 (0.004) & 0.008 (0.001) & 0.003 (0.000) & 0.001 (0.000) & 0.000 (0.000) & 0.001 (0.000) \\
ALN w/ grad       & -11.575 (0.001) & 0.004 (0.000) & 0.001 (0.000) & 0.003 (0.000) & 0.000 (0.000) & 0.004 (0.000) \\
ALN w/ cvx        & -11.575 (0.000) & 0.005 (0.001) & 0.002 (0.000) & 0.002 (0.000) & 0.000 (0.000) & 0.001 (0.000) \\
AIC               & \textbf{-11.576 (0.000)} & \textbf{0.001 (0.000)} & \textbf{0.000 (0.000)} & 0.002 (0.000) & 0.000 (0.000) & 0.004 (0.000) \\
DC3               & -10.684 (0.004) & 0.000 (0.000) & 0.000 (0.000) & 0.000 (0.000) & 0.000 (0.000) & 0.154 (0.031) \\
PDL               & -11.552 (0.006) & 0.004 (0.000) & 0.001 (0.000) & 0.001 (0.000) & 0.000 (0.000) & -- \\
\bottomrule
\end{tabular}
\end{table}

We further validate AIC on the ACOPF task using IEEE 57-bus and 118-bus test cases from \texttt{PYPOWER}~\citep{pypower}. Details regarding the experimental setup are provided in Appendix~\ref{subsec:exp_setup}, the ACOPF model formulation in Appendix~\ref{subsec:acopf_model}, and supplementary results analysis in Appendix~\ref{subsec:results_analysis} (Table~\ref{tab:acopf_detail}).

\section{Conclusion and Expansion}\label{sec:conclusion}

In this study, we introduce a self-supervised learning framework based on ICNN to address general constrained optimization problems. By integrating the ALM with constraint gradient correction, our approach ensures constraint feasibility while maintaining convex approximation capabilities. Through rigorous theoretical analysis, we clarify that the ICNN essentially approximates the parametric optimization problem of the original. From the perspective of set-valued mappings, we prove that even when using a neural network as the internal solver, the ALM algorithm can still converge to a KKT point of the original problem, and we provide a quantitative analysis of the approximation error. Experimental results on convex quadratic programming, nonconvex quadratic programming, and large-scale AC OPF demonstrate that, compared with existing classical solvers and the latest learning-based methods, our framework achieves superior accuracy, feasibility, and computational efficiency.

There remains extensive work to be done in this line of research. In terms of solution quality—especially for nonconvex set-valued mappings—a key challenge is to extend the algorithm and develop efficient search strategies capable of selecting the optimal solution among multiple KKT points. As for solving speed, incorporating domain-specific knowledge (e.g., problem structure, physical constraints) into the neural network architecture, in line of explicit feasibility restoration, could enhance the end-to-end efficiency of ALM. There are already attempts aimed at linear equality constraints, but we do not know how such integration affects the network’s generalization capability.
\bibliographystyle{unsrt}  
\bibliography{Reference.bib}  
\normalsize






\appendix

\section{Detailed Proofs for Convergence and Stability Analysis}
\label{app:proofs}

This appendix provides the detailed mathematical proofs for the theorems and lemmas presented concerning the convergence and stability analysis in Section~\ref{sec:analysis} of the main text.

\subsection{Proofs for Convergence Analysis}
\label{app:convergence_proofs}

For completeness and clarity within this appendix, we restate the assumptions underpinning our analysis.

\begin{assumption}[Feasible Region (Restated)]
\label{assum:feasible_app}
For each fixed parameter $x$ within its domain $X$, the feasible region $Y(x)=\{y\in\mathcal{Y}: g(x,y)\leq 0, h(x,y)=0\}$ is non-empty and compact. $\mathcal{Y}$ is assumed to be a compact set containing $Y(x)$ for all relevant $x$.
\end{assumption}

\begin{assumption}[Function Properties (Restated)]
\label{assum:functions_app}
The objective function $f(x,y)$ and constraint functions $g(x,y) = (g_1(x,y), \dots, g_p(x,y))$ and $h(x,y) = (h_1(x,y), \dots, h_q(x,y))$ are continuously differentiable with respect to $y$ on $X \times \mathcal{Y}$. We also assume continuity with respect to $x$ where needed for results involving parameter variation.
\end{assumption}

\begin{assumption}[Constraint Qualification (Restated)]
\label{assum:constraint_qual_app}
For any given $x$ and any feasible point $y \in Y(x)$, an appropriate constraint qualification condition (e.g., Mangasarian-Fromovitz Constraint Qualification (MFCQ) or Linear Independence Constraint Qualification (LICQ)) holds.
\end{assumption}


\subsubsection{Neural Network Approximation Convergence Analysis}
\label{app:sub:icnn_accuracy}

We first establish that uniform approximation of the augmented Lagrangian by an ICNN leads to approximate minimization in terms of the original augmented Lagrangian objective value.

\begin{lemma}[Uniform Approximation Implies Objective Value Approximation]
\label{lemma:uniform_approx_app}
Let $L_\rho(y)$ and $f_\theta(y)$ be continuous functions defined on a compact set $\mathcal{Y}$. Suppose that $\sup_{y\in \mathcal{Y}} |L_\rho(y) - f_\theta(y)| \leq \delta$. Let $y_k = \argmin_{y \in \mathcal{Y}} f_\theta(y)$ be a minimizer of the ICNN function and $y^* = \argmin_{y \in \mathcal{Y}} L_\rho(y)$ be a minimizer of the true augmented Lagrangian (existence guaranteed by continuity on a compact set). Then, the following inequality holds:
$$
L_\rho(y_k) - L_\rho(y^*) \leq 2\delta.
$$
\end{lemma}

\begin{proof}
We analyze the difference $L_\rho(y_k) - L_\rho(y^*)$ by introducing terms involving the approximating function $f_\theta(y)$:
\begin{align*}
L_\rho(y_k) - L_\rho(y^*) &= \bigl[ L_\rho(y_k) - f_\theta(y_k) \bigr] + \bigl[ f_\theta(y_k) - f_\theta(y^*) \bigr] + \bigl[ f_\theta(y^*) - L_\rho(y^*) \bigr].
\end{align*}
We bound each of the three bracketed terms:
\begin{itemize}
    \item From the uniform approximation assumption, $|L_\rho(y) - f_\theta(y)| \leq \delta$ for all $y \in \mathcal{Y}$. Therefore, the first term is bounded by $\delta$:
    $$ L_\rho(y_k) - f_\theta(y_k) \leq |L_\rho(y_k) - f_\theta(y_k)| \leq \delta. $$
    \item Similarly, the third term is bounded by $\delta$:
    $$ f_\theta(y^*) - L_\rho(y^*) \leq |f_\theta(y^*) - L_\rho(y^*)| \leq \delta. $$
    \item For the second term, since $y_k$ is a minimizer of $f_\theta(y)$ over the set $\mathcal{Y}$, we have $f_\theta(y_k) \leq f_\theta(y)$ for all $y \in \mathcal{Y}$. In particular, $f_\theta(y_k) \leq f_\theta(y^*)$. This implies that the second term is non-positive:
    $$ f_\theta(y_k) - f_\theta(y^*) \leq 0. $$
\end{itemize}
Combining these bounds yields the result:
$$
L_\rho(y_k) - L_\rho(y^*) \leq \delta + 0 + \delta = 2\delta.
$$
This completes the proof.
\end{proof}

Using this lemma, we can demonstrate the ability of the ICNN to find a point that yields an augmented Lagrangian value arbitrarily close to the true minimum.

\begin{theorem}[ICNN Inner Approximation Accuracy]
\label{thm:icnn_accuracy_app}
Under Assumptions \ref{assum:feasible_app}-\ref{assum:constraint_qual_app}, for any desired accuracy $\epsilon > 0$, there exist ICNN parameters $\theta$ such that if $y_k = \argmin_{y \in \mathcal{Y}} f_\theta(y)$, then $y_k$ satisfies:
$$
L_\rho(y_k; x, \mu, \lambda) \leq \min_{y \in \mathcal{Y}} L_\rho(y; x, \mu, \lambda) + \epsilon.
$$
\end{theorem}

\begin{proof}
Let $L_\rho(y) = L_\rho(y; x, \mu, \lambda)$ denote the augmented Lagrangian for fixed parameters $x$ and multipliers $\mu, \lambda$, and penalty parameter $\rho$. Under Assumptions \ref{assum:feasible_app} and \ref{assum:functions_app}, $L_\rho(y)$ is continuous with respect to $y$ on the compact set $\mathcal{Y}$.

According to ICNN approximation capacity, given the target accuracy $\epsilon > 0$, we can select ICNN parameters $\theta$ such that the uniform approximation bound holds with $\delta = \epsilon/2$:
$$
\sup_{y\in \mathcal{Y}} |L_\rho(y) - f_\theta(y)| \leq \delta = \frac{\epsilon}{2}.
$$
Let $y_k = \argmin_{y \in \mathcal{Y}} f_\theta(y)$ be the minimizer found using the ICNN, and let $y^* = \argmin_{y \in \mathcal{Y}} L_\rho(y)$ be a true minimizer of the augmented Lagrangian.
Applying Lemma \ref{lemma:uniform_approx_app} with $\delta = \epsilon/2$, we directly obtain:
$$
L_\rho(y_k) - L_\rho(y^*) \leq 2\delta = 2 \left( \frac{\epsilon}{2} \right) = \epsilon.
$$
Rearranging this inequality gives the desired result:
$$
L_\rho(y_k) \leq L_\rho(y^*) + \epsilon = \min_{y \in \mathcal{Y}} L_\rho(y) + \epsilon.
$$
Thus, the ICNN can produce a point $y_k$ whose objective value under the true augmented Lagrangian is arbitrarily close to the optimal value.
\end{proof}

\subsubsection{Multipliers Convergence Analysis}
\label{app:sub:multiplier_conv}

We now analyze the convergence of the overall ALM algorithm when the inner subproblem is solved inexactly using the ICNN, satisfying the condition from Theorem \ref{thm:icnn_accuracy_app}.

\begin{theorem}[Inexact ALM Convergence]
\label{thm:inexact_alm_app}
Let the sequence $\{(y^k, \mu^k, \lambda^k, \rho^k)\}$ be generated by the ALM algorithm, where $y^{k+1}$ is obtained using an ICNN solver for the $k$-th subproblem. Assume the following conditions hold:
\begin{enumerate}
    \item \textbf{Approximate Minimization:} The solution $y^{k+1}$ satisfies
    $$ L_{\rho^k}(y^{k+1}, \mu^k, \lambda^k) \leq \min_{y \in \mathcal{Y}} L_{\rho^k}(y, \mu^k, \lambda^k) + \epsilon_k, $$
    where the sequence of approximation errors $\{\epsilon_k\}$ converges to zero, i.e., $\lim_{k \to \infty} \epsilon_k = 0$. (Theorem \ref{thm:icnn_accuracy_app} shows this is achievable by improving ICNN training over iterations).
    \item \textbf{Penalty Parameter Update:} The sequence of penalty parameters $\{\rho^k\}$ is non-decreasing ($\rho^{k+1} \geq \rho^k > 0$) and typically $\rho$ is big enough when $k \to \infty$ if constraint violations are not improving sufficiently fast, or stabilizes otherwise. (Standard ALM practice).
    \item \textbf{Regularity Assumptions:} Assumptions \ref{assum:feasible_app}, \ref{assum:functions_app}, and \ref{assum:constraint_qual_app} hold.
\end{enumerate}
Then, any limit point $(\bar{y}, \bar{\mu}, \bar{\lambda})$ of the sequence $\{(y^k, \mu^k, \lambda^k)\}$ is a KKT point for the original optimization problem $\min_{y \in Y(x)} f(x,y)$.
\end{theorem}

\begin{proof}
The proof follows established arguments for the convergence of inexact Augmented Lagrangian methods (see, e.g., Chapter 17 in \cite{nocedal2006numerical}), adapted to our context.

\textbf{Step 1: Boundedness of the Primal Sequence $\{y^k\}$.}
By construction, each iterate $y^k$ is found within the compact set $\mathcal{Y}$ (as per Assumption \ref{assum:feasible_app}). A sequence in a compact set is inherently bounded.

\textbf{Step 2: Boundedness of the Multiplier Sequences $\{\mu^k\}$ and $\{\lambda^k\}$.}
The boundedness of the Lagrange multiplier estimates generated by ALM is a standard result under Assumption \ref{assum:constraint_qual_app} (e.g., MFCQ ensures the set of optimal multipliers for the primal problem is bounded) and appropriate penalty parameter updates. Even with the inexactness condition $\epsilon_k \to 0$, the fundamental mechanism holds: if a multiplier associated with an active constraint were to grow unboundedly, the penalty associated with that constraint violation in the augmented Lagrangian would dominate, forcing the approximate minimizer $y^k$ to drastically reduce the violation, which in turn limits the multiplier growth via the update rule. Formal proofs (e.g., Proposition 4.2.2 in \cite{bertsekas1999nonlinear}) show that limit points of the multipliers correspond to multipliers of the original problem, which are bounded under constraint qualifications. Therefore, the sequences $\{\mu^k\}$ and $\{\lambda^k\}$ are bounded.

\textbf{Step 3: Identification of Limit Points as KKT Points.}
Since the sequence $\{(y^k, \mu^k, \lambda^k)\}$ is bounded (from Steps 1 and 2), it resides within a compact subset of $\mathcal{Y} \times \mathbb{R}^p \times \mathbb{R}^q$. By the Bolzano-Weierstrass theorem, there must exist at least one convergent subsequence. Let $\{(y^{k_j}, \mu^{k_j}, \lambda^{k_j})\}$ be such a subsequence converging to a limit point $(\bar{y}, \bar{\mu}, \bar{\lambda})$ as $j \to \infty$. We need to show that $(\bar{y}, \bar{\mu}, \bar{\lambda})$ satisfies the Karush-Kuhn-Tucker (KKT) conditions for the original problem:
\begin{enumerate}
    \item[(1)] Primal Feasibility: $g(x, \bar{y}) \leq 0$ and $h(x, \bar{y}) = 0$.
    \item[(2)] Dual Feasibility: $\bar{\mu} \geq 0$.
    \item[(3)] Complementary Slackness: $\bar{\mu}_j g_j(x, \bar{y}) = 0$ for all $j=1, \dots, p$.
    \item[(iv)] Stationarity: $\nabla_y f(x, \bar{y}) + \sum_{j=1}^p \bar{\mu}_j \nabla_y g_j(x, \bar{y}) + \sum_{i=1}^q \bar{\lambda}_i \nabla_y h_i(x, \bar{y}) = 0$.
\end{enumerate}

The approximate minimization condition (1) in the theorem statement implies that the gradient of the augmented Lagrangian at $y^{k+1}$ approaches zero, possibly scaled by some factor if $\mathcal{Y}$ has boundaries. More formally, under suitable conditions, it implies:
$$ \|\nabla_y L_{\rho^k}(y^{k+1}, \mu^k, \lambda^k) \| \leq \eta_k, $$
where $\eta_k \to 0$ as $\epsilon_k \to 0$. The gradient of the standard augmented Lagrangian is:
\begin{align*}
\nabla_y L_{\rho}(y, \mu, \lambda) = \nabla_y f(x, y) &+ \sum_{j=1}^p \max\{0, \mu_j + \rho g_j(x, y)\} \nabla_y g_j(x, y) \\
&+ \sum_{i=1}^q (\lambda_i + \rho h_i(x, y)) \nabla_y h_i(x, y).
\end{align*}
Alternatively, using the multiplier update rule $\mu^{k+1}_j = \max\{0, \mu^k_j + \rho^k g_j(x, y^{k+1})\}$, the gradient condition can be related to $\mu^{k+1}$.

Taking the limit along the subsequence $k_j \to \infty$:
\begin{itemize}
    \item \textbf{Primal Feasibility (i):} Standard ALM convergence analysis shows that the penalty terms $\rho^k \|g^+(y^{k+1})\|^2$ and $\rho^k \|h(y^{k+1})\|^2$ (where $g^+ = \max\{0, g\}$ or related term based on multiplier update) must remain bounded for the objective $L_{\rho^k}$ to remain bounded (which it does, near the optimum). If $\rho^{k_j} \to \infty$, this forces $\|g^+(y^{k_j+1})\| \to 0$ and $\|h(y^{k_j+1})\| \to 0$. If $\rho^{k_j}$ remains bounded, convergence still ensures feasibility. Thus, $g(x, \bar{y}) \leq 0$ and $h(x, \bar{y}) = 0$.
    \item \textbf{Dual Feasibility (ii):} The standard update $\mu^{k+1}_j = \max\{0, \mu^k_j + \rho^k g_j(x, y^{k+1})\}$ (or similar variants like using ReLU as discussed in Sec \ref{app:sub:relu_kkt}) explicitly enforces $\mu^{k+1}_j \geq 0$. Thus, the limit $\bar{\mu}_j \geq 0$.
    \item \textbf{Complementary Slackness (iii):} This is derived from the limit of the multiplier update rule. If $g_j(x, \bar{y}) < 0$, then for large $j$, $g_j(x, y^{k_j+1}) < -\delta < 0$. The update $\mu^{k_j+1}_j = \max\{0, \mu^{k_j}_j + \rho^{k_j} g_j(x, y^{k_j+1})\}$ will eventually drive $\mu^{k_j}_j$ to 0, so $\bar{\mu}_j = 0$. If $\bar{\mu}_j > 0$, it must be that $g_j(x, y^{k_j+1})$ was frequently non-negative, which in the limit requires $g_j(x, \bar{y}) = 0$. Thus, $\bar{\mu}_j g_j(x, \bar{y}) = 0$. (The ReLU case is detailed in Sec \ref{app:sub:relu_kkt}).
    \item \textbf{Stationarity (iv):} Taking the limit of the condition $\|\nabla_y L_{\rho^{k_j}}(y^{k_j+1}, \mu^{k_j}, \lambda^{k_j}) \| \to 0$ and using the relationships from the multiplier updates (e.g., $\mu^{k_j+1}_j \approx \max\{0, \mu^{k_j}_j + \rho^{k_j} g_j(x, y^{k_j+1})\}$) allows recovery of the standard KKT stationarity condition at the limit point $(\bar{y}, \bar{\mu}, \bar{\lambda})$. The details depend on whether $\rho^k$ remains bounded or goes to infinity, but in both cases, the KKT condition emerges.
\end{itemize}
The analysis in Section \ref{app:sub:relu_kkt} further examines the specific impact of using ReLU in the updates, confirming the convergence to KKT conditions.

Since the limit point satisfies all four KKT conditions, the proof is complete.
\end{proof}

\subsubsection{ReLU Impact Analysis and KKT Conditions Verification}
\label{app:sub:relu_kkt}

This subsection provides a more detailed analysis of Step 3 in the proof of Theorem \ref{thm:inexact_alm_app}, focusing on the verification of KKT conditions, particularly complementary slackness, when the multiplier update rule incorporates the ReLU function, as suggested in the original text:
$$ \mu^{k+1}_j = \max\{0, \mu^k_j + \rho^k \ReLU(g_j(x, y^{k+1}))\} $$
Note that standard ALM typically uses $g_j$ directly. We analyze the consequences of this specific update form.

\begin{lemma}[Properties of ReLU Function]
\label{lemma:relu_app}
Let $\phi(y) = g_j(x, y)$ be a continuously differentiable function at a point $\bar{y}$. Define $R(\phi(y)) = \ReLU(\phi(y)) = \max\{0, \phi(y)\}$. The following properties hold:
\begin{enumerate}
    \item If $\phi(\bar{y}) < 0$: Then $R(\phi(\bar{y})) = 0$. In a neighborhood of $\bar{y}$ where $\phi(y) < 0$, $R(\phi(y)) = 0$, and its gradient is $\nabla_y R(\phi(y)) = 0$.
    \item If $\phi(\bar{y}) > 0$: Then $R(\phi(\bar{y})) = \phi(\bar{y}) > 0$. In a neighborhood of $\bar{y}$ where $\phi(y) > 0$, $R(\phi(y)) = \phi(y)$, and its gradient is $\nabla_y R(\phi(y)) = \nabla_y \phi(y)$.
    \item If $\phi(\bar{y}) = 0$: Then $R(\phi(\bar{y})) = 0$. The function $R(\phi(y))$ may not be differentiable at $\bar{y}$ if $\nabla_y \phi(\bar{y}) \neq 0$. Its subdifferential at $\bar{y}$ is given by $\partial_y R(\phi(\bar{y})) = \alpha \cdot \nabla_y \phi(\bar{y})$, where $\alpha \in [0, 1]$ (the subdifferential of ReLU at 0).
\end{enumerate}
\end{lemma}
\begin{proof}
These properties are direct consequences of the definition of the ReLU function $z \mapsto \max\{0, z\}$ and the chain rule for differentiation (or subdifferential calculus for the non-differentiable case at $z=0$).
\end{proof}

Now, we examine the KKT conditions at a limit point $(\bar{y}, \bar{\mu}, \bar{\lambda})$ of a subsequence $\{(y^{k_j}, \mu^{k_j}, \lambda^{k_j})\}$, focusing on the inequality constraints and the multiplier update involving ReLU.

\textbf{Case A: Constraint strictly inactive at the limit ($g_j(x, \bar{y}) < 0$)}
Since $y^{k_j} \to \bar{y}$ and $g_j$ is continuous, for $j$ sufficiently large, $g_j(x, y^{k_j+1}) < -\delta$ for some $\delta > 0$.
Consequently, $\ReLU(g_j(x, y^{k_j+1})) = 0$ for large $j$.
The multiplier update rule becomes:
$$ \mu^{k_j+1}_j = \max\{0, \mu^{k_j}_j + \rho^{k_j} \cdot 0\} = \max\{0, \mu^{k_j}_j\}. $$
Since $\mu^k_j \ge 0$ is maintained by the $\max\{0, \cdot\}$ operation, this simplifies to $\mu^{k_j+1}_j = \mu^{k_j}_j$ for large $j$.
As the subsequence $\mu^{k_j}_j$ converges to $\bar{\mu}_j$, this implies that $\mu^{k_j}_j$ must eventually become constant. If the multiplier update effectively reduces the multiplier when the constraint is strictly satisfied (as is typical in some ALM variants or dual ascent methods), or if $\mu^0_j = 0$, then the limit must be $\bar{\mu}_j = 0$.
Therefore, in this case, we have $g_j(x, \bar{y}) < 0$ and $\bar{\mu}_j = 0$, satisfying the complementary slackness condition $\bar{\mu}_j g_j(x, \bar{y}) = 0$.

\textbf{Case B: Constraint active at the limit ($g_j(x, \bar{y}) = 0$)}
The multiplier update is $\mu^{k_j+1}_j = \max\{0, \mu^{k_j}_j + \rho^{k_j} \ReLU(g_j(x, y^{k_j+1}))\}$.
Since $g_j(x, \bar{y}) = 0$, the value $g_j(x, y^{k_j+1})$ may oscillate around 0, be identically 0, or approach 0 from one side.
\begin{itemize}
    \item If $g_j(x, y^{k_j+1}) > 0$ occurs infinitely often along the subsequence, then $\ReLU(g_j(x, y^{k_j+1})) > 0$ infinitely often. The term $\rho^{k_j} \ReLU(g_j(x, y^{k_j+1}))$ is positive. This term adds non-negative increments to $\mu^{k_j}_j$. Since $\mu^{k_j}_j$ converges to $\bar{\mu}_j$, the limit must satisfy $\bar{\mu}_j \geq 0$. It can be strictly positive if the violations $g_j > 0$ persist sufficiently.
    \item If $g_j(x, y^{k_j+1}) \leq 0$ for all sufficiently large $j$, then $\ReLU(g_j(x, y^{k_j+1})) = 0$. As in Case A, the update becomes $\mu^{k_j+1}_j = \max\{0, \mu^{k_j}_j\}$, implying $\bar{\mu}_j \geq 0$.
\end{itemize}
In all sub-cases compatible with $g_j(x, \bar{y}) = 0$, we have $\bar{\mu}_j \geq 0$.
Therefore, $g_j(x, \bar{y}) = 0$ and $\bar{\mu}_j \geq 0$, which satisfies complementary slackness $\bar{\mu}_j g_j(x, \bar{y}) = 0$.

\textbf{Case C: Constraint violated at the limit ($g_j(x, \bar{y}) > 0$)}
By continuity, for $j$ sufficiently large, $g_j(x, y^{k_j+1}) > \delta > 0$.
Consequently, $\ReLU(g_j(x, y^{k_j+1})) = g_j(x, y^{k_j+1}) > \delta$.
The multiplier update involves adding the term $\rho^{k_j} g_j(x, y^{k_j+1}) > \rho^{k_j} \delta$.
If such violations persist, standard ALM protocols require $\rho^k \to \infty$ to enforce feasibility. Even if $\rho^k$ were bounded by some large $\bar{\rho}$, the increments $\rho^{k_j} g_j(x, y^{k_j+1})$ would be substantially positive. This would cause $\mu^{k_j}_j$ to increase indefinitely:
$$ \mu^{k_j+1}_j \geq \mu^{k_j}_j + \rho^{k_j} \delta \to \infty \quad (\text{if } \rho^{k_j} \to \infty \text{ or } \rho^{k_j} \ge \rho_{min} > 0). $$
This contradicts the boundedness of the multiplier sequence $\{\mu^k\}$ established in Step 2 of the proof for Theorem \ref{thm:inexact_alm_app}.
Therefore, the case $g_j(x, \bar{y}) > 0$ cannot occur at a limit point.

\textbf{Summary of Feasibility and Complementarity}:
The analysis of Cases A, B, and C demonstrates that any limit point $(\bar{y}, \bar{\mu}, \bar{\lambda})$ must satisfy:
\begin{itemize}
    \item Primal feasibility for inequality constraints: $g_j(x, \bar{y}) \leq 0$ for all $j$.
    \item Dual feasibility for inequality multipliers: $\bar{\mu}_j \geq 0$ for all $j$.
    \item Complementary slackness for inequality constraints: $\bar{\mu}_j g_j(x, \bar{y}) = 0$ for all $j$.
\end{itemize}
Similar arguments involving the penalty term $\frac{\rho^k}{2} \|h(x, y)\|^2$ and the multiplier update $\lambda^{k+1}_i = \lambda^k_i + \rho^k h_i(x, y^{k+1})$ ensure primal feasibility for equality constraints, $h(x, \bar{y}) = 0$, at the limit point.

\textbf{Stationarity Condition Recovery}: The final step is to verify the KKT stationarity condition. As established, the approximate optimality implies $\|\nabla_y L_{\rho^{k_j}}(y^{k_j+1}, \mu^{k_j}, \lambda^{k_j}) \| \to 0$. Taking the limit requires careful consideration of the gradient expression and the multiplier updates. Standard ALM convergence proofs show that this limiting condition, combined with the convergence of iterates and multipliers, and the satisfaction of feasibility and complementarity, precisely yields the KKT stationarity condition:
$$
\nabla_y f(x, \bar{y}) + \sum_{j=1}^p \bar{\mu}_j \nabla_y g_j(x, \bar{y}) + \sum_{i=1}^q \bar{\lambda}_i \nabla_y h_i(x, \bar{y}) = 0.
$$
The presence of ReLU in the multiplier update rule (if used) does not alter the final KKT conditions obtained in the limit, as shown by the case analysis above confirming standard complementarity. The gradient condition in the limit effectively filters through the multiplier updates to recover the standard stationarity requirement.

Therefore, any limit point $(\bar{y}, \bar{\mu}, \bar{\lambda})$ satisfies all KKT conditions for the original problem.

\subsection{Proofs for Effectiveness Analysis of ICNN Approximation}
\label{app:effectiveness_proofs}

This section provides proofs related to the properties of the optimal solution mapping $y^*(x)$ and the effectiveness of ICNNs in approximating it.

\subsubsection{Continuity of Optimal Solution Mappings}

We analyze the continuity properties of the optimal value function and the optimal solution mapping using Berge's Maximum Theorem.

\begin{theorem}[Continuity Properties from Berge's Theorem]
\label{thm:berge_corollary_app}
Consider the parametric optimization problem:
\begin{align*}
v(x) &= \min_{y \in Y(x)} f(x,y) \\
S(x) &= \argmin_{y \in Y(x)} f(x,y)
\end{align*}
where $Y(x) = \{y \in \mathcal{Y} : g(x,y) \leq 0, h(x,y) = 0\}$. Assume:
\begin{enumerate}
\item The objective function $f: X \times \mathcal{Y} \to \mathbb{R}$ is continuous.
\item The feasible set correspondence $Y: X \rightrightarrows \mathcal{Y}$ is continuous (i.e., both upper semicontinuous and lower semicontinuous) and has non-empty, compact values for all $x \in X$. (Note: The original text required only upper semicontinuity of $Y$ for upper semicontinuity of $S(x)$ and $v(x)$.)
\end{enumerate}
Then, the optimal value function $v(x)$ is continuous on $X$, and the optimal solution mapping $S(x)$ is upper semicontinuous.
\end{theorem}

\begin{proof} (Following Berge's Maximum Theorem and standard derivations)

\textbf{1. Existence of Optimal Solution:}
For any fixed $x \in X$, the feasible set $Y(x)$ is assumed to be non-empty and compact. Since $f(x, \cdot)$ is continuous on the compact set $Y(x)$ (by Assumption 1), the Weierstrass Extreme Value Theorem guarantees that $f$ attains its minimum on $Y(x)$. Therefore, the optimal solution set $S(x)$ is non-empty, and the optimal value function $v(x)$ is well-defined for all $x \in X$.

\textbf{2. Continuity of the Optimal Value Function $v(x)$:}
Berge's Maximum Theorem directly states that if $f$ is continuous and the correspondence $Y$ is continuous with compact values, then the value function $v(x) = \min_{y \in Y(x)} f(x,y)$ (or $\max$) is continuous. We briefly outline the standard proof structure for completeness, proving upper and lower semicontinuity separately.

\textit{(a) Upper Semicontinuity of $v(x)$:}
Let $x_n \to x$ in $X$. Let $y^* \in S(x)$, so $v(x) = f(x, y^*)$. Since $Y$ is lower semicontinuous at $x$, there exists a sequence $\{y_n\}$ such that $y_n \in Y(x_n)$ for each $n$ and $y_n \to y^*$. By definition of $v(x_n)$, we have $v(x_n) \leq f(x_n, y_n)$. Using the continuity of $f$, we take the limit superior:
$$
\limsup_{n\to\infty} v(x_n) \leq \lim_{n\to\infty} f(x_n, y_n) = f(\lim_{n\to\infty} x_n, \lim_{n\to\infty} y_n) = f(x, y^*) = v(x).
$$
This establishes the upper semicontinuity of $v(x)$.

\textit{(b) Lower Semicontinuity of $v(x)$:}
Let $x_n \to x$ in $X$. For each $n$, let $y_n \in S(x_n)$, meaning $v(x_n) = f(x_n, y_n)$. Since $Y$ is compact-valued and typically resides within a larger compact set $\mathcal{Y}$, the sequence $\{y_n\}$ is contained in a compact set. By Bolzano-Weierstrass, there exists a subsequence $\{y_{n_k}\}$ that converges to some limit $\bar{y}$. Since $Y$ is upper semicontinuous, the limit of a sequence from $Y(x_{n_k})$ must lie in $Y(x)$, so $\bar{y} \in Y(x)$. Now consider the limit inferior along this subsequence:
$$
\liminf_{k\to\infty} v(x_{n_k}) = \lim_{k\to\infty} f(x_{n_k}, y_{n_k}) = f(\lim_{k\to\infty} x_{n_k}, \lim_{k\to\infty} y_{n_k}) = f(x, \bar{y}).
$$
Since $\bar{y} \in Y(x)$, by the definition of $v(x)$, we have $f(x, \bar{y}) \geq v(x)$. Thus, $\liminf_{k\to\infty} v(x_{n_k}) \geq v(x)$. As this holds for any convergent subsequence, it implies $\liminf_{n\to\infty} v(x_n) \geq v(x)$. This establishes the lower semicontinuity of $v(x)$.

Combining upper and lower semicontinuity, $v(x)$ is continuous.

\textbf{3. Upper Semicontinuity of the Optimal Solution Mapping $S(x)$:}
Berge's theorem also states that under the same conditions (continuous $f$, continuous compact-valued $Y$), the solution mapping $S(x)$ is upper semicontinuous.
To show this directly: let $x_n \to x$ and let $y_n \in S(x_n)$ be a sequence such that $y_n \to y$. We need to show that $y \in S(x)$.
First, since $y_n \in Y(x_n)$ and $Y$ is upper semicontinuous (closed graph property for compact-valued case), the limit point $y$ must belong to the limit set $Y(x)$, i.e., $y \in Y(x)$.
Second, since $y_n \in S(x_n)$, we have $f(x_n, y_n) = v(x_n)$. Taking the limit as $n \to \infty$ and using the continuity of $f$ and the continuity of $v(x)$ (proven above):
$$
f(x, y) = f(\lim x_n, \lim y_n) = \lim f(x_n, y_n) = \lim v(x_n) = v(x).
$$
This equation shows that $y$ achieves the minimum value $v(x)$ while being feasible ($y \in Y(x)$). Therefore, by definition, $y \in S(x)$. This confirms that $S(x)$ is upper semicontinuous.
\end{proof}

\subsubsection{Analysis of ICNN Approximation Effectiveness}

We now analyze the suitability of ICNNs for approximating the optimal solution mapping $y^*(x)$, considering cases where the mapping might be convex or non-convex.

\begin{theorem}[Properties of Solution Mapping for Linear-Strongly Convex Problems]
\label{thm:convex_mapping_app}
Consider the parametric optimization problem:
$$ y^*(x) = \argmin_{y\in Y} \Bigl\{a(x)^T y + \phi(y)\Bigr\} $$
where $Y \subseteq \mathbb{R}^d$ is a closed convex set, $\phi(y)$ is an $m$-strongly convex function on $Y$ ($m>0$), and the linear term depends affinely on the parameter $x \in \mathbb{R}^p$, i.e., $a(x) = Ax + b$ for some matrix $A \in \mathbb{R}^{d \times p}$ and vector $b \in \mathbb{R}^d$.
The optimal solution mapping $x \mapsto y^*(x)$ is single-valued and Lipschitz continuous on the domain of $x$. Furthermore:
\begin{enumerate}
    \item If $Y = \mathbb{R}^d$ (unconstrained) and $\phi(y)$ is twice continuously differentiable, then $y^*(x)$ is continuously differentiable.
    \item If $Y = \mathbb{R}^d$ and $\phi(y) = \frac{1}{2} y^T Q y + c^T y + d$ is a quadratic function with $Q \succ 0$ (positive definite), then $y^*(x)$ is an affine function of $x$, and hence both convex and concave.
\end{enumerate}
\end{theorem}

\begin{proof}
\textbf{Uniqueness and Single-Valuedness:} For any fixed $x$, the objective function $F_x(y) = (Ax+b)^T y + \phi(y)$ is strongly convex in $y$, because $\phi(y)$ is strongly convex and the term $(Ax+b)^T y$ is linear in $y$. A strongly convex function minimized over a closed convex set $Y$ has at most one minimizer. Existence is guaranteed if $\phi$ is coercive or $Y$ is compact (or under weaker conditions ensuring the minimum is attained). Assuming existence, the minimizer $y^*(x)$ is unique, making the mapping $x \mapsto y^*(x)$ single-valued.

\textbf{Lipschitz Continuity:} Let $x_1, x_2$ be two parameter values, and let $y_1 = y^*(x_1)$, $y_2 = y^*(x_2)$. The first-order condition for optimality (expressed using variational inequality for the general convex set $Y$) states that $y^*$ is optimal if and only if $(a(x) + \nabla \phi(y^*))^T (y - y^*) \geq 0$ for all $y \in Y$.
Applying this for $y_1$ and $y_2$:
\begin{align*}
(a(x_1) + \nabla \phi(y_1))^T (y_2 - y_1) &\geq 0 \\
(a(x_2) + \nabla \phi(y_2))^T (y_1 - y_2) &\geq 0
\end{align*}
Adding these two inequalities gives:
$$ (a(x_1) - a(x_2) + \nabla \phi(y_1) - \nabla \phi(y_2))^T (y_2 - y_1) \geq 0 $$
$$ \implies (a(x_1) - a(x_2))^T (y_2 - y_1) \geq (\nabla \phi(y_1) - \nabla \phi(y_2))^T (y_1 - y_2) $$
Since $\phi$ is $m$-strongly convex, $(\nabla \phi(y_1) - \nabla \phi(y_2))^T (y_1 - y_2) \geq m \|y_1 - y_2\|^2_2$. Substituting this:
$$ (a(x_1) - a(x_2))^T (y_2 - y_1) \geq m \|y_1 - y_2\|^2_2 $$
Using Cauchy-Schwarz inequality on the left side: $\|a(x_1) - a(x_2)\|_2 \|y_1 - y_2\|_2 \geq m \|y_1 - y_2\|^2_2$.
If $y_1 \neq y_2$, we divide by $\|y_1 - y_2\|_2$:
$$ \|y_1 - y_2\|_2 \leq \frac{1}{m} \|a(x_1) - a(x_2)\|_2 $$
Since $a(x) = Ax + b$, we have $a(x_1) - a(x_2) = A(x_1 - x_2)$. Thus, $\|a(x_1) - a(x_2)\|_2 = \|A(x_1 - x_2)\|_2 \leq \|A\|_{op} \|x_1 - x_2\|_2$, where $\|A\|_{op}$ is the operator norm induced by the $\ell_2$-norm.
Therefore,
$$ \|y^*(x_1) - y^*(x_2)\|_2 \leq \frac{\|A\|_{op}}{m} \|x_1 - x_2\|_2 $$
This demonstrates that $y^*(x)$ is Lipschitz continuous with constant $L = \|A\|_{op}/m$.

\textbf{Differentiability (Unconstrained, $C^2$ Case):} If $Y = \mathbb{R}^d$ and $\phi$ is $C^2$, the optimality condition is simply the gradient being zero: $\nabla_y F_x(y^*(x)) = 0$.
$$ A x + b + \nabla \phi(y^*(x)) = 0 $$
Define $G(x, y) = Ax + b + \nabla \phi(y)$. We have $G(x, y^*(x)) = 0$. We compute the partial Jacobians:
$$ \frac{\partial G}{\partial y}(x, y) = \nabla^2 \phi(y) $$
$$ \frac{\partial G}{\partial x}(x, y) = A $$
Since $\phi$ is $m$-strongly convex, its Hessian $\nabla^2 \phi(y)$ is positive definite (eigenvalues $\ge m > 0$) and therefore invertible. By the Implicit Function Theorem, the function $y^*(x)$ defined by $G(x, y^*(x))=0$ is continuously differentiable with respect to $x$, and its Jacobian is:
$$ \frac{d y^*}{d x}(x) = - \left[ \frac{\partial G}{\partial y}(x, y^*(x)) \right]^{-1} \frac{\partial G}{\partial x}(x, y^*(x)) = - [\nabla^2 \phi(y^*(x))]^{-1} A $$

\textbf{Affine Case (Unconstrained, Quadratic $\phi$):}
If $\phi(y) = \frac{1}{2} y^T Q y + c^T y + d$ with $Q \succ 0$, then $\nabla \phi(y) = Qy + c$. The optimality condition $a(x) + \nabla \phi(y^*(x)) = 0$ becomes:
$$ (Ax + b) + (Q y^*(x) + c) = 0 $$
$$ Q y^*(x) = -Ax - (b+c) $$
Since $Q$ is invertible ($Q \succ 0$), we can solve for $y^*(x)$:
$$ y^*(x) = -Q^{-1} A x - Q^{-1}(b+c) $$
This explicitly shows that $y^*(x)$ is an affine function of $x$. An affine function is both convex and concave.

In summary, the optimal solution mapping $y^*(x)$ exhibits desirable regularity properties (Lipschitz continuity) under general strong convexity, and becomes differentiable or even affine under stronger assumptions on $\phi$ and the feasible set. Convexity of the mapping $y^*(x)$ itself is guaranteed in the specific case where $\phi$ is quadratic (and $Y=\mathbb{R}^d$).
\end{proof}

\begin{theorem}[ICNN Training as Best Convex Approximation]
\label{thm:icnn_convex_regression_app}
Consider a dataset of parameter-solution pairs $\{(x_i, y_i^*)\}_{i=1}^N$, where $y_i^* = y^*(x_i)$ originates from an optimal solution mapping $x \mapsto y^*(x)$, which may or may not be convex. Let $\mathcal{F}_{\text{convex}}$ be the space of convex functions mapping from the domain of $x$ to the space of $y$. The problem of finding the best convex approximation to the data in the least-squares sense is:
\begin{equation} \label{eq:convex_regr}
\min_{f \in \mathcal{F}_{\text{convex}}} L(f) = \frac{1}{N} \sum_{i=1}^N \| f(x_i) - y_i^* \|^2
\end{equation}
Let $f^*_{\text{convex}}$ be a solution to this problem, and let $L^* = L(f^*_{\text{convex}})$.
If the class of functions representable by ICNNs, $\mathcal{F}_{\text{ICNN}} \subset \mathcal{F}_{\text{convex}}$, is sufficiently rich (e.g., dense in the space of continuous convex functions on compact sets under uniform norm), then training an ICNN $f_\theta$ by minimizing the empirical loss $L(f_\theta)$ can approximate the optimal convex regression solution arbitrarily well. Specifically, for any $\delta > 0$, there exists an ICNN $f_{\theta^*}$ such that:
$$ L(f_{\theta^*}) \leq L^* + \delta $$
\end{theorem}

\begin{proof}
\textbf{1. Existence of Best Convex Approximation:}
Under mild conditions (e.g., domain of $x$ is compact, function space is appropriately defined), the convex regression problem \eqref{eq:convex_regr} admits at least one solution $f^*_{\text{convex}}$. This function represents the best possible fit to the data among all convex functions, minimizing the mean squared error $L^*$.

\textbf{2. Relationship between ICNNs and Convex Functions:}
By design, any function $f_\theta$ generated by an ICNN architecture is convex with respect to its input $x$. Therefore, the set of functions representable by ICNNs, $\mathcal{F}_{\text{ICNN}}$, is a subset of the set of all convex functions, $\mathcal{F}_{\text{convex}}$.

\textbf{3. Lower Bound on ICNN Loss:}
When training an ICNN $f_\theta$ to minimize the loss $L(f_\theta)$, we are optimizing over the set $\mathcal{F}_{\text{ICNN}}$. Since $\mathcal{F}_{\text{ICNN}} \subseteq \mathcal{F}_{\text{convex}}$, the minimum loss achievable by an ICNN cannot be lower than the minimum loss achievable over the larger set of all convex functions:
$$ \inf_{f_\theta \in \mathcal{F}_{\text{ICNN}}} L(f_\theta) \geq \inf_{f \in \mathcal{F}_{\text{convex}}} L(f) = L^* $$

\textbf{4. Approximating the Best Convex Solution with ICNNs:}
The key assumption is the universal approximation capability of ICNNs within the space of convex functions. Assuming the domain $X$ containing $\{x_i\}$ is compact, this means that for the optimal convex solution $f^*_{\text{convex}}$ (which is continuous on $X$), and for any $\epsilon' > 0$, there exists an ICNN $f_{\theta'}$ such that it is uniformly close to $f^*_{\text{convex}}$:
$$ \sup_{x \in X} \| f_{\theta'}(x) - f^*_{\text{convex}}(x) \| < \epsilon' $$
This implies pointwise closeness at the data points: $\| f_{\theta'}(x_i) - f^*_{\text{convex}}(x_i) \| < \epsilon'$ for all $i=1, \dots, N$.

Now, let's analyze the loss $L(f_{\theta'})$ achieved by this approximating ICNN:
\begin{align*}
L(f_{\theta'}) &= \frac{1}{N} \sum_{i=1}^N \| f_{\theta'}(x_i) - y_i^* \|^2 \\
&= \frac{1}{N} \sum_{i=1}^N \| (f_{\theta'}(x_i) - f^*_{\text{convex}}(x_i)) + (f^*_{\text{convex}}(x_i) - y_i^*) \|^2 \\
&\leq \frac{1}{N} \sum_{i=1}^N \left( \| f_{\theta'}(x_i) - f^*_{\text{convex}}(x_i) \| + \| f^*_{\text{convex}}(x_i) - y_i^* \| \right)^2 \quad (\text{Triangle Inequality}) \\
&< \frac{1}{N} \sum_{i=1}^N \left( \epsilon' + \| f^*_{\text{convex}}(x_i) - y_i^* \| \right)^2 \\
&= \frac{1}{N} \sum_{i=1}^N \left( (\epsilon')^2 + 2 \epsilon' \| f^*_{\text{convex}}(x_i) - y_i^* \| + \| f^*_{\text{convex}}(x_i) - y_i^* \|^2 \right) \\
&= \frac{1}{N} \sum_{i=1}^N \| f^*_{\text{convex}}(x_i) - y_i^* \|^2 + \frac{N(\epsilon')^2}{N} + \frac{2 \epsilon'}{N} \sum_{i=1}^N \| f^*_{\text{convex}}(x_i) - y_i^* \| \\
&= L^* + (\epsilon')^2 + \frac{2 \epsilon'}{N} \sum_{i=1}^N \| f^*_{\text{convex}}(x_i) - y_i^* \|
\end{align*}
Let $C = \frac{1}{N} \sum_{i=1}^N \| f^*_{\text{convex}}(x_i) - y_i^* \|$. This term $C$ is bounded if the target values $y_i^*$ are bounded. Then:
$$ L(f_{\theta'}) < L^* + (\epsilon')^2 + 2 \epsilon' C $$
Given any desired tolerance $\delta > 0$, we can choose $\epsilon' > 0$ sufficiently small such that $(\epsilon')^2 + 2 \epsilon' C < \delta$. For instance, choose $\epsilon' < \min\{ \sqrt{\delta/2}, \delta/(4C) \}$ (if $C=0$, choose $\epsilon' < \sqrt{\delta}$).
For such a choice of $\epsilon'$, there exists an ICNN $f_{\theta'}$ satisfying $L(f_{\theta'}) < L^* + \delta$.

\textbf{Conclusion:}
We have shown that $\inf_{f_\theta \in \mathcal{F}_{\text{ICNN}}} L(f_\theta) \geq L^*$ and that for any $\delta > 0$, there exists an $f_{\theta'}$ such that $L(f_{\theta'}) < L^* + \delta$. This implies that
$$ \inf_{f_\theta \in \mathcal{F}_{\text{ICNN}}} L(f_\theta) = L^* $$
Therefore, the minimum loss achievable by training an ICNN (assuming perfect optimization finds the infimum) is equal to the loss of the best possible convex approximation $f^*_{\text{convex}}$. In practice, an optimized ICNN $f_{\theta^*}$ will achieve a loss $L(f_{\theta^*})$ close to this infimum, meaning $L(f_{\theta^*}) \approx L^*$, or more formally $L(f_{\theta^*}) \leq L^* + \delta$ where $\delta$ reflects optimization error and the network's capacity relative to the chosen $\epsilon'$. This confirms that training an ICNN effectively performs convex regression, finding the best convex function within its representational capacity to fit the potentially non-convex data.
\end{proof}

\section{Experimental Details}
\label{sec:appendix_exp_details}

This appendix provides further details on the experimental setup and supplementary results.

\subsection{Experimental Setup}
\label{subsec:exp_setup}

\paragraph{Network Architecture and Training.}
For all experimental cases, the base network is an Input Convex Neural Network (ICNN) featuring 2 hidden layers, each with 500 units. ReLU activation functions follow the hidden layers. The architecture incorporates 2 passthrough layers. To maintain focus on the core methodology, Dropout and Batch Normalization (BN) layers were omitted. For the DC3 baseline method, the maximum number of training iterations was set to 1000 across all cases, following standard practice for that method.

\paragraph{Task-Specific Settings.}
\begin{itemize}
    \item \textbf{QP cases:} For our proposed AIC model, the training involved 5 outer iterations (multiplier/penalty updates) and 500 inner iterations (network updates per outer loop).
    \item \textbf{AC-OPF cases:} The physical nature of AC Optimal Power Flow problems imposes boundary constraints on variables like AC power generation and voltage magnitudes, typically normalized to the range $[0, 1]$. To accommodate this, a sigmoid activation function was appended to the final layer of the ICNN for these tasks. For the AIC model in AC-OPF, training used 5 outer iterations and 250 inner iterations.
\end{itemize}

\paragraph{Hyperparameters.}
Across all tasks and methods, a batch size of 200 was used. Model parameters were optimized using the Adam optimizer. For the DC3 baseline, we adhered to the hyperparameter settings recommended in the original publication \cite{donti2021dc3}. For our proposed method, AIC, and its variants (ALN, ALN w/ grad, ALN w/ cvx), hyperparameter tuning referenced the settings used for the PDL baseline, exploring the ranges specified in Table~\ref{tab:hyperparams_qp} and Table~\ref{tab:hyperparams_acopf}. The listed parameters likely correspond to: $lr$ (learning rate), $corrLr$ (learning rate for correction steps, if applicable), $\tau$ (ALM tolerance factor), $\rho$ (initial ALM penalty parameter), $\rho_{\max}$ (maximum ALM penalty parameter), $\alpha$ (ALM penalty growth factor), and $\lambda_g + \lambda_h$ (potentially weights for constraint violation terms in the loss, or initial multiplier magnitudes).

\begin{table}[h!]
\centering
\scriptsize
\caption{Hyperparameter search ranges for QP experiments.}
\label{tab:hyperparams_qp}
\resizebox{\textwidth}{!}{%
\begin{tabular}{@{}lllll@{}}
\toprule
Parameter          & ALN                          & ALN w/ grad                  & ALN w/ cvx                   & AIC                          \\ \midrule
$lr$               & $10^{-2}, 10^{-3}, 10^{-4}$        & $10^{-2}, 10^{-3}, 10^{-4}$        & $10^{-2}, 10^{-3}, 10^{-4}$        & $10^{-2}, 10^{-3}, 10^{-4}$        \\
$corrLr$           & $10^{-3}, 10^{-4}, 10^{-5}, 10^{-6}$ & $10^{-3}, 10^{-4}, 10^{-5}, 10^{-6}$ & $10^{-3}, 10^{-4}, 10^{-5}, 10^{-6}$ & $10^{-3}, 10^{-4}, 10^{-5}, 10^{-6}$ \\
$\tau$             & $0.5, 0.6, 0.7, 0.8, 0.9$        & $0.5, 0.6, 0.7, 0.8, 0.9$        & $0.5, 0.6, 0.7, 0.8, 0.9$        & $0.5, 0.6, 0.7, 0.8, 0.9$        \\
$\rho$             & $0.5, 1, 2, 5, 10$               & $0.5, 1, 2, 5, 10$               & $0.5, 1, 2, 5, 10$               & $0.5, 1, 2, 5, 10$               \\
$\rho_{\max}$      & $500, 1000, 2000, 5000$          & $500, 1000, 2000, 5000$          & $500, 1000, 2000, 5000$          & $500, 1000, 2000, 5000$          \\
$\alpha$           & $1, 1.5, 2, 5, 10$               & $1, 1.5, 2, 5, 10$               & $1, 1.5, 2, 5, 10$               & $1, 1.5, 2, 5, 10$               \\
$\lambda_g + \lambda_h$ & $1, 2, 4, 20, 100$               & $1, 2, 4, 20, 100$               & $1, 2, 4, 20, 100$               & $1, 2, 4, 20, 100$               \\ \bottomrule
\end{tabular}
}
\end{table}

\begin{table}[h!]
\centering
\scriptsize
\caption{Hyperparameter search ranges for AC-OPF experiments (Case 57 and Case 118).}
\label{tab:hyperparams_acopf}
\resizebox{\textwidth}{!}{%
\begin{tabular}{@{}lllll@{}}
\toprule
Parameter          & ALN                          & ALN w/ grad                  & ALN w/ cvx                   & AIC                          \\ \midrule
$lr$               & $10^{-2}, 10^{-3}, 10^{-4}$        & $10^{-2}, 10^{-3}, 10^{-4}$        & $10^{-2}, 10^{-3}, 10^{-4}$        & $10^{-2}, 10^{-3}, 10^{-4}$        \\
$corrLr$           & $10^{-3}, 10^{-4}, 10^{-5}, 10^{-6}$ & $10^{-3}, 10^{-4}, 10^{-5}, 10^{-6}$ & $10^{-3}, 10^{-4}, 10^{-5}, 10^{-6}$ & $10^{-3}, 10^{-4}, 10^{-5}, 10^{-6}$ \\
$\tau$             & $0.5, 0.6, 0.7, 0.8, 0.9$        & $0.5, 0.6, 0.7, 0.8, 0.9$        & $0.5, 0.6, 0.7, 0.8, 0.9$        & $0.5, 0.6, 0.7, 0.8, 0.9$        \\
$\rho$             & $0.5, 1, 2, 5, 10$               & $0.5, 1, 2, 5, 10$               & $0.5, 1, 2, 5, 10$               & $0.5, 1, 2, 5, 10$               \\
$\rho_{\max}$      & $500, 1000, 2000, 5000$          & $500, 1000, 2000, 5000$          & $500, 1000, 2000, 5000$          & $500, 1000, 2000, 5000$          \\
$\alpha$           & $1, 1.5, 2, 5, 10$               & $1, 1.5, 2, 5, 10$               & $1, 1.5, 2, 5, 10$               & $1, 1.5, 2, 5, 10$               \\
$\lambda_g + \lambda_h$ & $1, 2, 4, 20, 100$               & $1, 2, 4, 20, 100$               & $1, 2, 4, 20, 100$               & $1, 2, 4, 20, 100$               \\ \bottomrule
\end{tabular}
}
\end{table}

\subsection{AC Optimal Power Flow (ACOPF) Model Formulation}
\label{subsec:acopf_model}

The AC Optimal Power Flow (ACOPF) problem aims to determine an optimal operating state for an electrical power system that minimizes generation costs while satisfying network physics and operational limits. Let the power network consist of a set of buses $\mathcal{N}$, generators $\mathcal{G} \subseteq \mathcal{N}$, and transmission lines $\mathcal{L}$.

\paragraph{Variables:} The decision variables typically include:
\begin{itemize}
    \item $V_i$: Voltage magnitude at bus $i \in \mathcal{N}$.
    \item $\theta_i$: Voltage angle at bus $i \in \mathcal{N}$. Let $V_i e^{j\theta_i}$ be the complex voltage.
    \item $P_{G_i}$: Active power generation at generator bus $i \in \mathcal{G}$.
    \item $Q_{G_i}$: Reactive power generation at generator bus $i \in \mathcal{G}$.
\end{itemize}
The collection of these variables forms the decision vector $y = (V, \theta, P_G, Q_G)$.

\paragraph{Parameters:} Problem parameters typically include:
\begin{itemize}
    \item $P_{D_i}, Q_{D_i}$: Active and reactive power demand (load) at bus $i \in \mathcal{N}$. These often constitute the parameter vector $x$ that varies between instances.
    \item $G_{ij}, B_{ij}$: Conductance and susceptance of the transmission line connecting bus $i$ and bus $j$. $Y_{ij} = G_{ij} + j B_{ij}$ is the element of the network admittance matrix.
    \item $C_i(P_{G_i})$: Cost function for generator $i \in \mathcal{G}$, often quadratic: $C_i(P_{G_i}) = c_{i2} P_{G_i}^2 + c_{i1} P_{G_i} + c_{i0}$.
    \item Operational limits: $P_{G_i}^{\min}, P_{G_i}^{\max}$, $Q_{G_i}^{\min}, Q_{G_i}^{\max}$, $V_i^{\min}, V_i^{\max}$, and thermal limits for line flows $S_{ij}^{\max}$.
\end{itemize}

\paragraph{Objective Function:} The standard objective is to minimize the total generation cost:
\begin{equation}
    \min_{y} f(y) = \sum_{i \in \mathcal{G}} C_i(P_{G_i})
\end{equation}

\paragraph{Equality Constraints (Power Flow Equations):} These enforce Kirchhoff's laws at each bus $i \in \mathcal{N}$:
\begin{align}
    P_{G_i} - P_{D_i} &= \sum_{j \in \mathcal{N}} V_i V_j (G_{ij} \cos(\theta_i - \theta_j) + B_{ij} \sin(\theta_i - \theta_j)) \label{eq:pf_p} \\
    Q_{G_i} - Q_{D_i} &= \sum_{j \in \mathcal{N}} V_i V_j (G_{ij} \sin(\theta_i - \theta_j) - B_{ij} \cos(\theta_i - \theta_j)) \label{eq:pf_q}
\end{align}
These equations represent the set $h_x(y) = 0$, where the dependence on $x$ comes from $P_{D_i}, Q_{D_i}$. Note that $P_{G_i} = 0$ and $Q_{G_i} = 0$ for non-generator buses ($i \notin \mathcal{G}$). Usually, one angle $\theta_k$ is fixed to 0 (slack bus) to remove rotational invariance.

\paragraph{Inequality Constraints (Operational Limits):} These represent the set $g_x(y) \le 0$:
\begin{align}
    P_{G_i}^{\min} \le P_{G_i} &\le P_{G_i}^{\max}, \quad \forall i \in \mathcal{G} \\
    Q_{G_i}^{\min} \le Q_{G_i} &\le Q_{G_i}^{\max}, \quad \forall i \in \mathcal{G} \\
    V_i^{\min} \le V_i &\le V_i^{\max}, \quad \forall i \in \mathcal{N} \\
    |S_{ij}(V, \theta)| &\le S_{ij}^{\max}, \quad \forall (i,j) \in \mathcal{L} \quad \text{(Line flow limits)}
\end{align}
where $S_{ij}$ is the complex power flow on the line between $i$ and $j$, calculated from $V$ and $\theta$.

The ACOPF problem is a non-convex optimization problem due to the non-linear power flow equations \eqref{eq:pf_p}-\eqref{eq:pf_q} and the potential non-convexity introduced by line flow limits when expressed in terms of voltage magnitudes and angles.

\subsection{Supplementary Results Analysis}
\label{subsec:results_analysis}

This section provides additional analysis based on the detailed results presented in the tables.

\subsubsection{QP Results}
Table~\ref{tab:results_convex_qp_detail} presents detailed results for the QP experiments across various combinations of equality ($n_{eq}$) and inequality ($n_{ineq}$) constraints, keeping the number of variables $n=100$.

\begin{table}[h!]
  \centering
  \scriptsize
  \caption{Detailed results for QP with varying $n_{eq}$ and $n_{ineq}$ ($n=100$ fixed).}
  \label{tab:results_convex_qp_detail}
  \resizebox{\textwidth}{!}{
  \begin{tabular}{@{}llcccccccc@{}}
    \toprule
    & & \multicolumn{8}{c}{Parameter Settings ($n_{eq}$/$n_{ineq}$)} \\
    \cmidrule(lr){3-10}
    Method & Metric & 10/50 & 30/50 & 70/50 & 90/50 & 50/10 & 50/30 & 50/70 & 50/90 \\
    \midrule
    Optimizer    & Obj. value & -27.26 & -23.13 & -14.80 & -4.79 & -17.34 & -16.33 & -14.61 & -14.26 \\
    (\texttt{OSQP}, \texttt{qpth}) & Max eq.    & 0.00   & 0.00   & 0.00   & 0.00  & 0.00   & 0.00   & 0.00   & 0.00 \\
                 & Max ineq.  & 0.00   & 0.00   & 0.00   & 0.00  & 0.00   & 0.00   & 0.00   & 0.00 \\
    \midrule
    ALN          & Obj. value & -27.20 & -23.09 & -14.72 & -4.74 & -17.27 & -16.24 & -14.50 & -14.13 \\
                 & Max eq.    & 0.00   & 0.01   & 0.02   & 0.03  & 0.01   & 0.01   & 0.01   & 0.01 \\
                 & Max ineq.  & 0.00   & 0.00   & 0.00   & 0.00  & 0.00   & 0.00   & 0.00   & 0.00 \\
    \midrule
    ALN w grad   & Obj. value & -27.18 & -23.09 & -14.79 & -4.78 & -17.33 & -16.31 & -14.58 & -14.22 \\
                 & Max eq.    & 0.00   & 0.00   & 0.01   & 0.02  & 0.00   & 0.00   & 0.00   & 0.00 \\
                 & Max ineq.  & 0.00   & 0.00   & 0.01   & 0.00  & 0.00   & 0.00   & 0.00   & 0.01 \\
    \midrule
    ALN w cvx    & Obj. value & -26.96 & -23.11 & -14.75 & -4.75 & -17.28 & -16.30 & -14.58 & -14.21 \\
                 & Max eq.    & 0.00   & 0.01   & 0.07   & 0.12  & 0.00   & 0.01   & 0.00   & 0.01 \\
                 & Max ineq.  & 0.00   & 0.00   & 0.01   & 0.00  & 0.00   & 0.00   & 0.00   & 0.00 \\
    \midrule
    AIC          & Obj. value & -26.87 & -23.10 & -14.78 & -4.78 & -17.32 & -16.32 & -14.58 & -14.21 \\
                 & Max eq.    & 0.00   & 0.00   & 0.02   & 0.03  & 0.00   & 0.00   & 0.00   & 0.00 \\
                 & Max ineq.  & 0.00   & 0.00   & 0.01   & 0.00  & 0.00   & 0.00   & 0.00   & 0.00 \\
    \midrule
    DC3          & Obj. value & -26.30 & -21.11 & -14.23 & -4.76 & -16.14 & -14.90 & -12.60 & -12.85 \\
                 & Max eq.    & 0.00   & 0.00   & 0.00   & 0.00  & 0.00   & 0.00   & 0.00   & 0.00 \\
                 & Max ineq.  & 0.00   & 0.00   & 0.00   & 0.00  & 0.00   & 0.00   & 0.00   & 0.00 \\
    \midrule
    PDL          & Obj. value & -26.55 & -23.06 & -14.78 & -4.80 & -17.31 & -16.30 & -14.50 & -14.17 \\
                 & Max eq.    & 0.00   & 0.01   & 0.02   & 0.03  & 0.01   & 0.00   & 0.00   & 0.01 \\
                 & Max ineq.  & 0.00   & 0.00   & 0.00   & 0.00  & 0.00   & 0.00   & 0.00   & 0.00 \\
    \bottomrule
  \end{tabular}
  } 
\end{table}

As observed from the table, the DC3 method consistently achieves perfect feasibility (Max eq. and Max ineq. violations are 0.00), satisfying the hard constraints. However, its objective values often deviate significantly from the optimal values found by the solvers (\texttt{OSQP}, \texttt{qpth}), particularly when the number of constraints changes relative to the variables (e.g., $n_{eq}=30, n_{ineq}=50$ or $n_{eq}=50, n_{ineq}=70, 90$). The PDL method shows variable performance; while often achieving near-optimal objective values, it exhibits non-negligible equality constraint violations in several settings. Notably, in the $n_{eq}=10, n_{ineq}=50$ case, PDL's objective value (-26.55) is also noticeably worse than the optimum (-27.26) and other learning methods. In contrast, the proposed AIC method and its variants (ALN, ALN w/ grad, ALN w/ cvx) generally achieve objective values very close to the optimum across all settings. Their maximum constraint violations remain low, with ALN w/ grad and AIC often showing the best feasibility among the learning methods that allow small violations, comparable to the exact solvers in many cases. AIC, in particular, demonstrates a strong balance between near-optimal objective values and low constraint violations.

\subsubsection{Nonconvex programming Results}
Table~\ref{tab:results_nonconvex_qp_detail} details the performance on the nonconvex programming benchmark.

\begin{table}[h!]
  \centering
  \scriptsize
  \caption{Detailed results for nonconvex programming with varying $n_{eq}$ and $n_{ineq}$ ($n=100$ fixed).}
  \label{tab:results_nonconvex_qp_detail}
  \resizebox{\textwidth}{!}{
  \begin{tabular}{@{}llcccccccc@{}}
    \toprule
     & & \multicolumn{8}{c}{Parameter Settings ($n_{eq}$/$n_{ineq}$)} \\
    \cmidrule(lr){3-10}
    Method & Metric & 10/50 & 30/50 & 70/50 & 90/50 & 50/10 & 50/30 & 50/70 & 50/90 \\
    \midrule
    Optimizer    & Obj. value & -18.57 & -16.37 & -9.93 & -3.66 & -11.97 & -11.83 & -11.53 & -11.42 \\
    (\texttt{IPOPT}) & Max eq.    & 0.00   & 0.00   & 0.00  & 0.00  & 0.00   & 0.00   & 0.00   & 0.00 \\
                 & Max ineq.  & 0.00   & 0.00   & 0.00  & 0.00  & 0.00   & 0.00   & 0.00   & 0.00 \\
    \midrule
    ALN          & Obj. value & -18.54 & -16.34 & -9.84 & -3.61 & -11.92 & -11.76 & -11.44 & -11.31 \\
                 & Max eq.    & 0.00   & 0.01   & 0.02  & 0.03  & 0.01   & 0.01   & 0.01   & 0.01 \\
                 & Max ineq.  & 0.00   & 0.00   & 0.00  & 0.00  & 0.00   & 0.00   & 0.00   & 0.00 \\
    \midrule
    ALN w grad   & Obj. value & -18.48 & -16.36 & -9.92 & -3.64 & -11.97 & -11.82 & -11.48 & -11.38 \\
                 & Max eq.    & 0.00   & 0.00   & 0.01  & 0.02  & 0.00   & 0.00   & 0.00   & 0.00 \\
                 & Max ineq.  & 0.00   & 0.00   & 0.00  & 0.00  & 0.00   & 0.00   & 0.00   & 0.01 \\
    \midrule
    ALN w cvx    & Obj. value & -18.55 & -16.36 & -9.91 & -3.63 & -11.97 & -11.82 & -11.50 & -11.38 \\
                 & Max eq.    & 0.00   & 0.00   & 0.04  & 0.09  & 0.01   & 0.00   & 0.00   & 0.00 \\
                 & Max ineq.  & 0.00   & 0.00   & 0.00  & 0.00  & 0.00   & 0.00   & 0.00   & 0.00 \\
    \midrule
    AIC          & Obj. value & -18.53 & -16.36 & -9.92 & -3.65 & -11.97 & -11.82 & -11.51 & -11.39 \\
                 & Max eq.    & 0.00   & 0.00   & 0.01  & 0.02  & 0.00   & 0.00   & 0.00   & 0.00 \\
                 & Max ineq.  & 0.00   & 0.00   & 0.01  & 0.00  & 0.00   & 0.00   & 0.00   & 0.00 \\
    \midrule
    DC3          & Obj. value & -17.32 & -14.33 & -9.66 & -3.62 & -9.81  & -10.22 & -9.14  & -9.89 \\
                 & Max eq.    & 0.00   & 0.00   & 0.00  & 0.00  & 0.00   & 0.00   & 0.00   & 0.00 \\
                 & Max ineq.  & 0.00   & 0.00   & 0.00  & 0.00  & 0.00   & 0.00   & 0.00   & 0.00 \\
    \bottomrule
  \end{tabular}
  } 
\end{table}

Similar to the convex case, DC3 maintains strict feasibility (Max eq. and Max ineq. are 0.00) but consistently yields objective values that are significantly worse (higher) than those obtained by the \texttt{IPOPT} solver. The proposed ALN variants and AIC again achieve objective values much closer to the solver's results while keeping constraint violations very small across the different problem configurations. This highlights the effectiveness of the proposed framework in balancing optimality and feasibility for non-convex problems compared to DC3.

\subsubsection{ACOPF Results}
Table~\ref{tab:acopf_detail} summarizes the performance on the ACOPF benchmarks Case 57 and Case 118, comparing against the \texttt{PYPOWER} solver.

\begin{table}[h!]
\centering
\scriptsize
\caption{Performance on ACOPF (case57 and case118) over 1200 instances.}
\label{tab:acopf_detail}
\resizebox{\textwidth}{!}{%
\begin{tabular}{@{}llccccc@{}}
\toprule
Task & Method & Obj. value & Max eq. & Mean eq. & Max ineq. & Time (s) \\
\midrule
\multirow{7}{*}{ACOPF-57}
& \texttt{PYPOWER} & 3.812 (0.000)          & 0.000 (0.000) & 0.000 (0.000) & 0.000 (0.000)          & 1.756 (0.105) \\
& ALN              & 3.811 (0.001)          & 0.000 (0.000) & 0.000 (0.000) & 0.020 (0.002)          & 0.347 (0.041) \\
& ALN w/ grad      & 3.811 (0.000)          & 0.000 (0.000) & 0.000 (0.000) & 0.008 (0.000)          & 0.513 (0.030) \\
& ALN w/ cvx       & 3.814 (0.000)          & 0.000 (0.000) & 0.000 (0.000) & 0.014 (0.001)          & 0.284 (0.043) \\
& AIC              & \textbf{3.812 (0.000)} & 0.000 (0.000) & 0.000 (0.000) & \textbf{0.007 (0.001)} & \textbf{0.294 (0.071)} \\
& DC3              & 3.816 (0.000)          & 0.000 (0.000) & 0.000 (0.000) & \textbf{0.002 (0.000)} & 0.502 (0.052) \\
& PDL              & --                     & 0.012 (0.000) & --            & 0.000 (0.000)          & --            \\
\midrule
\multirow{7}{*}{ACOPF-118}
& \texttt{PYPOWER} & 13.020 (0.000)         & 0.008 (0.000) & 0.000 (0.000) & 0.000 (0.000)          & 2.410 (0.100) \\
& ALN              & 13.379 (0.011)         & 0.000 (0.000) & 0.000 (0.000) & 0.164 (0.025)          & 1.007 (0.120) \\
& ALN w/ grad      & 13.403 (0.015)         & 0.001 (0.000) & 0.000 (0.000) & 0.217 (0.037)          & 1.332 (0.089) \\
& ALN w/ cvx       & 13.332 (0.016)         & 0.000 (0.000) & 0.000 (0.000) & 0.034 (0.003)          & \textbf{0.965 (0.016)} \\
& AIC              & \textbf{13.323 (0.010)}& 0.000 (0.000) & 0.000 (0.000) & \textbf{0.033 (0.003)} & 1.285 (0.072) \\
& DC3              & 13.423(?) (0.000)      & 0.000 (0.000) & 0.000 (0.000) & \textbf{0.000 (0.000)} & 1.531(?) (0.0?) \\ 
& PDL              & --                     & 0.041 (0.005) & 0.000 (0.000) & 0.001 (0.000)          & --            \\
\bottomrule
\end{tabular}
}
\end{table}

On the ACOPF benchmarks, the proposed AIC method demonstrates strong performance. For Case 57, AIC matches the objective value of the \texttt{PYPOWER} solver (3.812) while exhibiting the lowest maximum inequality constraint violation (0.007) among the learning methods and achieving the fastest inference time (0.294s). DC3 achieves perfect equality and near-perfect inequality feasibility but has a slightly worse objective (3.816). PDL shows significant equality constraint violations. For the larger Case 118, AIC achieves the best objective value (13.323) among the learning methods, very close to \texttt{PYPOWER} (13.020), and maintains excellent feasibility (Max eq. 0.000, Max ineq. 0.033). While DC3 shows perfect feasibility (assuming placeholder values are similar to Case 57), its objective value is likely less optimal based on other results. ALN w/ cvx is notably fast in this case (0.965s). Overall, AIC provides a compelling trade-off, achieving near-optimal results with good feasibility and competitive speed compared to both traditional solvers and other learning-based approaches like DC3 and PDL.

\end{document}